\RequirePackage{fix-cm}
\documentclass[smallextended]{svjour3}       %
\smartqed  %
\usepackage{hyperref}
\usepackage{natbib}

\newcommand{\fs}{\lambda_1} %
\newcommand{\nfs}{\lambda_2} %
\newcommand{\fd}{\mu_2} %
\newcommand{\nfd}{\mu_1} %
\newcommand{\sr}{\theta_s} %
\newcommand{\dr}{\theta_d} %
\newcommand{\fsprob}{\gamma_s} %
\newcommand{\fdprob}{\gamma_d} %
\newcommand{\1}[1]{\mathbf{1}_{\{#1\}}} %

\newcommand{\hnu}[1]{\textcolor{black}{#1}}
\newcommand{\fcu}[1]{\textcolor{black}{#1}}
\newcommand{\cyu}[1]{\textcolor{black}{#1}}

\usepackage{tkz-euclide}
\usetkzobj{all}
\usetikzlibrary{arrows.meta}

\newcommand{\gettikzxy}[3]{%
  \tikz@scan@one@point\pgfutil@firstofone#1\relax
  \edef#2{\the\pgf@x}%
  \edef#3{\the\pgf@y}%
}

\usepackage{hyperref}
\usepackage{url}
\usepackage{graphicx}
\usepackage{tikz}
\usepackage{amsfonts}
\usepackage{amsmath}
\usepackage[margin=2.3cm]{geometry}
\usepackage{etoolbox}
\newcommand*{\affaddr}[1]{#1} %
\newcommand*{\affmark}[1][*]{\textsuperscript{#1}}

\usepackage{color}
\newcommand{\revision}[1]{\color{black}{#1}}
\journalname{Queueing Systems}
\begin{document}

\title{Matching Queues with Reneging: a Product Form Solution}

\titlerunning{Matching Queues with Reneging}        %

\author{Francisco Castro\affmark[*] \and Hamid Nazerzadeh\affmark[*] \and Chiwei Yan\affmark[*]}

\authorrunning{Castro, Nazerzadeh, and Yan} %

\institute{Francisco Castro \at
              Anderson School of Management \\
              University of California, Los Angeles\\
              \email{francisco.castro@anderson.ucla.edu}           %
           \and
           Hamid Nazerzadeh \at
           Marshall School of Business \\
           University of Southern California \\
           \email{nazerzad@usc.edu}
           \and
           Chiwei Yan \at
           Department of Industrial and Systems Engineering \\
           University of Washington, Seattle\\
           \email{chiwei@uw.edu}\\[2mm]
           \affaddr{\affmark[*]This work is done when the authors are with Uber Technologies, Inc.}
}

\date{\empty}%

\maketitle

\vspace{-5em}
\begin{abstract}
Motivated by growing applications in two-sided markets, we study a parallel matching queue with reneging. Demand and supply units arrive to the system and are matched in an FCFS manner according to a compatibility graph specified by an N-system. If they cannot be matched upon arrival, they queue and may abandon the system  as time goes by. We derive explicit product forms of the steady state distributions of this system by identifying a partial balance condition.

\end{abstract}

\section{Introduction}\label{intro}
We consider a parallel matching queue where supply and demand arrive randomly over time. Supply and demand units are matched, in an FCFS manner, according to the compatibility graph depicted in Figure \ref{fig:n-system-intro}. Our model is a variation of the widely studied N-system \citep{green1985queueing,adan2009exact,visschers2012product,adan2018fcfs,zhan2018many}. There are two types of supply, {\em flexible}, indexed by 1, and {\em inflexible}, indexed by 2. Correspondingly, we refer to the demand type that can only be served by the flexible supply as type 1, and the demand units that can be served by both flexible and inflexible supply as type 2. A key feature of our model is that the supply and demand can be impatient. More specifically, we analyze two systems. In the first system, demand units leave if they are not matched upon arrival. In the second system, demand units may queue---namely, if there is no supply waiting upon their arrival to the system, demand units start queuing and are matched to a compatible supply as soon as one becomes available. In both of these systems, if not immediately matched, the supply joins a queue, but would abandon the system at rate $\sr$. In the second system, the demand queues if not matched immediately but abandons at rate $\dr$.

\begin{figure}[htbp]
\centering
\scalebox{0.7}{\begin{tikzpicture}[baseline=0pt]

\node[rotate=90] at (-10.5,3){\Large \textbf{Supply Side}};
\draw[line width=0.6mm] (-7,4.5)--(-4,4.5);
\draw[line width=0.6mm] (-7,3.5)--(-4,3.5);
\draw[line width=0.6mm] (-4,3.5)--(-4,4.5);
\draw[line width=0.6mm,-] (-3.6,4.0)--(-1,4.0);
\draw[line width=0.6mm,-] (-3.6,4.0)--(-1,2.0);
\draw[line width=0.6mm,-] (-3.6,2.0)--(-1.0,2.0);

\draw[line width=0.6mm,->] (-8.5,4.0)--(-7.5,4.0) node at(-10.25+1,4) {\Large $\lambda_1$};
\draw[line width=0.6mm,->] (-8.5,2)--(-7.5,2) node at(-10.25+1,2) {\Large $\lambda_2$};
\draw[line width=0.6mm,->] (-5.5,4.0)--(-5.5,3.0) node at (-5.15,3.0) {\Large $\sr$};
\draw[line width=0.6mm,->] (-5.5,2.0)--(-5.5,1.0) node at (-5.15,1.0) {\Large $\sr$};

\def\sf{-2}
\draw[line width=0.6mm] (-7,4.5+\sf)--(-4,4.5+\sf);
\draw[line width=0.6mm] (-7,3.5+\sf)--(-4,3.5+\sf);
\draw[line width=0.6mm] (-4,3.5+\sf)--(-4,4.5+\sf);



\node[rotate=90] at (6,3){\Large \textbf{Demand Side}};
\def\sf{-2}
\def\sg{4.5}
\draw[line width=0.6mm] (7-\sg,4.5+\sf)--(4-\sg,4.5+\sf);
\draw[line width=0.6mm] (7-\sg,3.5+\sf)--(4-\sg,3.5+\sf);
\draw[line width=0.6mm] (4-\sg,3.5+\sf)--(4-\sg,4.5+\sf);

\draw[line width=0.6mm] (7-\sg,4.5)--(4-\sg,4.5);
\draw[line width=0.6mm] (7-\sg,3.5)--(4-\sg,3.5);
\draw[line width=0.6mm] (4-\sg,3.5)--(4-\sg,4.5);

\def\sf{2}
\draw[line width=0.6mm,->] (2+\sf,2.0)--(1+\sf,2.0) node at (3.0+\sf,2) {\Large $\mu_2$};
\draw[line width=0.6mm,->] (2+\sf,4.0)--(1+\sf,4.0) node at (3.0+\sf,4) {\Large $\mu_1$};
\draw[line width=0.6mm,->] (-5.5+6.5,4.0)--(-5.5+6.5,3.0) node at (-5.15+6.5,3.0) {\Large $\dr$};
\draw[line width=0.6mm,->] (-5.5+6.5,2.0)--(-5.5+6.5,1.0) node at (-5.15+6.5,1.0) {\Large $\dr$};

\end{tikzpicture}}
\caption{N-system matching queue.}
\label{fig:n-system-intro}
\end{figure}
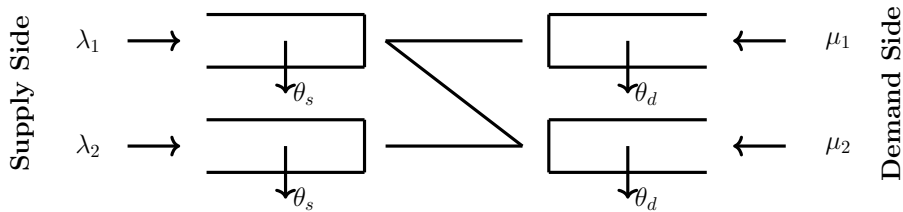

The main contribution of this work is to provide \emph{product form solutions} for the steady state probabilities for the systems above. To the extent of our knowledge, this is the first exact analysis for parallel matching queues with reneging. Our analysis builds on the parsimonious state space representation developed by \cite{adan2009exact}, \cite{visschers2012product} and \cite{adan2014skill}. This representation merges the inflexible and flexible queues into a single queue with two parts: a known part of inflexible agents at the beginning of the queue, and an unknown part of mixed agents at the tail of the queue. As we discuss in Section~\ref{sec:imp-reneg}, incorporating reneging introduces challenges to the analysis because the number of agents in the known part impacts the state dynamics of the unknown part. This leads us to a non-trivial product form solution with interdependent terms. 
The main idea behind our solution approach is to identify an appropriate \emph{partial balance} condition that the steady state probabilities would  satisfy and then to leverage this condition to obtain  exact product form solutions.

Our work is motivated in part by growing applications of matching queuing systems in the gig economy, including online job and task marketplaces such as Upwork and TaskRabbit as well as  applications in healthcare such as blood transfusions and organ transplant markets. In all these applications, {compatibility} plays a significant role in the underlying matching process. 
As another example, let us consider an application in ride-sharing. Uber's app for drivers provides a feature often referred to as the ``destination mode,'' in which 
drivers can enter a specific destination.\footnote{See \url{https://www.uber.com/us/en/drive/basics/driver-destinations/}} Then the platform would only send them trip requests from the riders that are going toward the same destination. For a given destination, our model captures drivers who prefer that destination as inflexible supply and drivers with no preferred  destination as flexible supply. Similarly, the riders going to that destination would be considered as type 2 because they can be served by both driver types. Our analysis provides insights on how the rate adoption of this feature (ratio of flexible to inflexible supply) would impact the marketplace.

\paragraph{\bf Related work} In the following, we briefly discuss the lines of research to which our paper makes a contribution.

\medskip
\textbf{Product form.} 
Since the classic work of \cite{jackson1957networks,jackson1963jobshop}, 
researchers have pursued product form solutions for a diverse range of settings such as systems with redundant requests \citep{gardner2016queueing}, infinite bipartite matching system \citep{adan2017reversibility}, and stochastic matching on general graphs \citep{moyal2017product}.  We refer readers to a recent overview by \cite{gardner2020product}.
Closest to our work are papers that study the parallel FCFS queues in the context of production and service systems.
In particular, \cite{adan2009exact} consider 
an allocation policy that randomizes jobs when, upon arrival, they see idle servers. In this context, they develop the first product form solution for the steady state distribution. 
\cite{visschers2012product} consider a generalization of \cite{adan2009exact} with multiple machines and job types and derive a product form solution for the steady state distribution. More recently, \cite{adan2014skill} introduced the FCFS-ALIS---first come first served, assign longest idle server---policy and obtain a product form solution for the steady state probabilities. 
\cite{adan2018fcfs} provide a unified view of several related systems, including the aforementioned FCFS-ALIS parallel queue, a redundancy queue, and an FCFS matching queue. They show that these three systems are closely connected to a directed infinite bipartite matching model and present product form solutions. 
Our work contributes to this literature by providing a product form solution for the stationary probabilities of an FCFS matching queue with reneging, which encompasses a wider range of applications.
Incorporating reneging is non-trivial as it leads to a different kind of product form which cannot be written as a product of independent terms, as discussed in Section \ref{sec:one-sided}.

\medskip
\textbf{Reneging.} 
\cite{kaplan1988public} presented the first fluid approximation of an overloaded queue with reneging, with application in public housing assignment. \cite{talreja2008fluid} strengthened the analysis by proposing fluid models under different matching graphs. Motivated by organ transplant applications, \cite{zenios1999modeling} studies a multi-class system in which patients may renege, whereas organs do not line up in the system. He provides asymptotic performance measures for a randomized allocation policy that effectively decouples the system into multiple independent M/M/1 queues with reneging.  \cite{boxma2011new}
proposes a double-sided queue to investigate the organ allocation problem, in which both patients and organs may abandon. In related work,  \cite{afeche2014double} consider a double-sided queue with reneging, but in which arrivals come in batches. Using level-crossing techniques, these works derive expressions for steady state performance metrics. In contrast, our focus is to obtain an \emph{exact} analysis of the steady state behavior of systems with reneging.

\medskip
\paragraph{\bf Organization} The remainder of the paper is organized as follows. In Section \ref{sec:one-sided}, we first analyze the one-sided system in which only supply queues and demand leaves the system if they cannot be matched immediately upon arrival. We then, in Section \ref{sec:two-sided}, extend the analysis to the two-sided setting in which both demand and supply queue. We conclude in Section \ref{sec:conclusion}.
 
\section{One-sided System}\label{sec:one-sided}

In this section, we analyze a one-sided matching queuing system, under FCFS policy, where demand arrivals have zero patience and leave the system if they cannot be matched immediately. Supply queues up in the system and each unit stays in the system for an i.i.d. and exponentially distributed time with rate $\sr$. 
The arrival of each type of supply and demand follows a Poisson process, with rates denoted by $\fs$ ($\fd$) and $\nfs$ ($\nfd$). Recall that flexible supply, indexed by 1, can be matched to either type 1 or type 2 demand, and inflexible supply, indexed by 2, can be matched only to type 2 demand.
We use abandonment and reneging interchangeably in the rest of the paper.

\subsection{State Space Representation}
\label{subsec:state_space_representation}

A natural way to model the state of the system is to keep track of all supply arrivals in the order of their arrival times. However, this state representation is exponential in size. In this paper, similar to \cite{adan2009exact} and \cite{visschers2012product}, we consider an alternative parsimonious representation: \cyu{we arrange different types of supply into a single queue according to their arrival times, and then partition them into two parts: a \emph{known} part with inflexible supply at the head of the queue and an \emph{unknown} part with a mixed types of supply at the tail of the queue}. The state space is $\mathfrak{s}=\left\{(m,n) \mid m,n\in\mathbb{Z}_{\ge0}\right\}$ where $m$ is the number of known inflexible supplies at the beginning of the queue, and $n$ is the remaining number of supplies with unknown type, i.e. they could be either flexible or inflexible. 
At any state $(m,n)$, an arrival of flexible demand will consume a supply at the beginning of the queue, if there is any; an arrival of inflexible demand has to go through an independent sequence of trials to scan unknown supply in their order of arrival until finding a flexible one. Note that the latter occurs with probability $\gamma_s=\lambda_1/(\lambda_1 + \lambda_2)$, that is, $\gamma_s$ is the probability of a supply arrival being flexible. The scanned inflexible supply will be added into the known part at the beginning of the queue. The state transitions are depicted in Figure \ref{fig:one-sided-state-diag}. \hnu{Similar to  \cite{visschers2012product}, it is easy to observe that this state space representation is Markovian.}
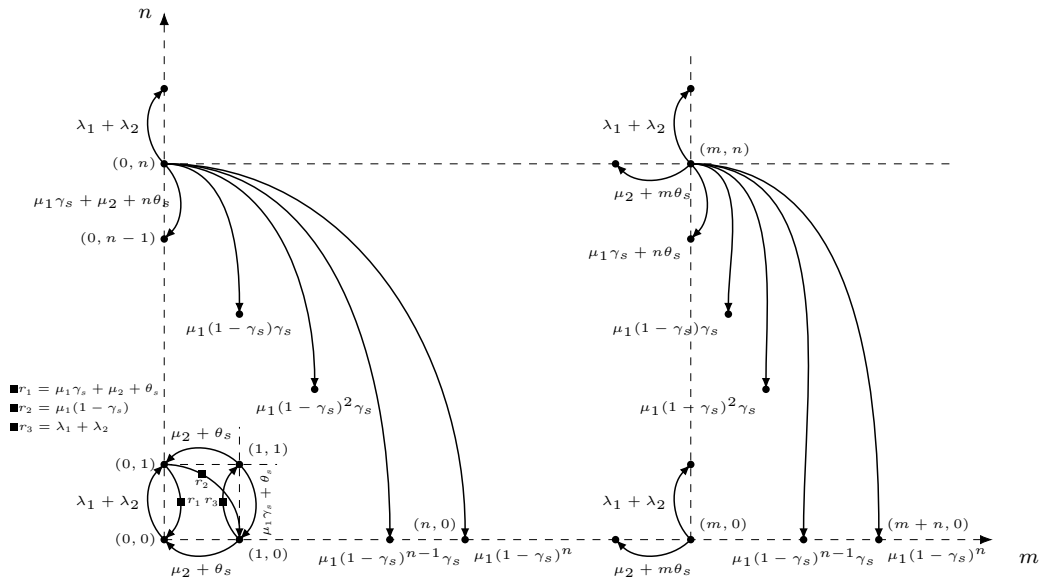
\begin{figure}[h]
\centering
\begin{tikzpicture}[baseline=0pt]


\coordinate[label =  left: \text{\tiny $(0,0)$}] (A) at (0,0);
\coordinate[label =  below right: \text{\tiny $(1,0)$}] (B) at (1,0);
\coordinate[label =  left: \text{\tiny $(0,1)$}] (C) at (0,1);
\coordinate[label =  above right: \text{\tiny $(1,1)$}] (D) at (1,1);
\node[rotate=90,font=\fontsize{4.5}{3.75}\selectfont] at (1.35,0.5) { $\mu_1\gamma_s+\sr$};

\node at (A)[circle,fill,inner sep=1pt]{};
\node at (B)[circle,fill,inner sep=1pt]{};
\node at (C)[circle,fill,inner sep=1pt]{};
\node at (D)[circle,fill,inner sep=1pt]{};

\node at (0.22,0.5)[fill,inner sep=1.5pt]{};
\node[font=\fontsize{4.5}{3.75}\selectfont] at (0.4,0.5){ $r_1$};
\node at (-2,2)[fill,inner sep=1.5pt]{};
\node[font=\fontsize{4.5}{3.75}\selectfont] at (-1,2){ $r_1=\mu_1\gamma_s+\mu_2 +\sr$};

\node at (0.5,0.87)[fill,inner sep=1.5pt]{};
\node[font=\fontsize{4.5}{3.75}\selectfont] at (0.5,0.74){ $r_2$};
\node at (-2,1.75)[fill,inner sep=1.5pt]{};
\node[font=\fontsize{4.5}{3.75}\selectfont] at (-1.2,1.75){ $r_2=\mu_1(1-\gamma_s)$};

\node at (0.78,0.5)[fill,inner sep=1.5pt]{};
\node[font=\fontsize{4.5}{3.75}\selectfont] at (0.63,0.5){ $r_3$};
\node at (-2,1.5)[fill,inner sep=1.5pt]{};
\node[font=\fontsize{4.5}{3.75}\selectfont] at (-1.33,1.5){ $r_3=\lambda_1+\lambda_2$};

\draw[line width=0.2mm,-{Latex[length=1.5mm]}](D) to [out=135,in=45] node[above]{\tiny $\mu_2 +\sr$} (C);

\draw[line width=0.2mm,-{Latex[length=1.5mm]}](A) to [out=135,in=225] node[left]{\tiny $\fs+\nfs$} (C);

\draw[line width=0.2mm,-{Latex[length=1.5mm]}](C) to [out=-45,in=45] node[right]{} (A);

\draw[line width=0.2mm,-{Latex[length=1.5mm]}](C) to [out=0,in=90] node[left]{} (B);

\draw[line width=0.2mm,-{Latex[length=1.5mm]}](B) to [out=225,in=-45] node[below]{\tiny $\fd+\sr$} (A);

\draw[line width=0.2mm,-{Latex[length=1.5mm]}](D) to [out=-45,in=45] node[left]{} (B);
\draw[line width=0.2mm,-{Latex[length=1.5mm]}](B) to [out=135,in=225] node[left]{} (D);

\coordinate[label =  above right: \text{\tiny $(m,0)$}] (A1) at (7,0);
\node at (A1)[circle,fill,inner sep=1pt]{};
\coordinate[] (A2) at (6,0);
\node at (A2)[circle,fill,inner sep=1pt]{};
\coordinate[] (A3) at (7,1);
\node at (A3)[circle,fill,inner sep=1pt]{};

\draw[line width=0.2mm,-{Latex[length=1.5mm]}](A1) to [out=225,in=-45] node[below]{\tiny $\fd+m\sr$} (A2);
\draw[line width=0.2mm,-{Latex[length=1.5mm]}](A1) to [out=135,in=225] node[left]{\tiny $\fs+\nfs$} (A3);

\coordinate[] (B0) at (0,6);
\node at (B0)[circle,fill,inner sep=1pt]{};
\coordinate[label =  left: \text{\tiny $(0,n)$}] (B1) at (0,5);
\node at (B1)[circle,fill,inner sep=1pt]{};
\coordinate[label =  left: \text{\revision{\tiny $(0,n-1)$}}] (B2) at (0,4);
\coordinate[] (B2) at (0,4);
\node at (B2)[circle,fill,inner sep=1pt]{};
\coordinate[] (B3) at (1,3);
\node at (B3)[circle,fill,inner sep=1pt]{};
\coordinate[] (B4) at (2,2);
\node at (B4)[circle,fill,inner sep=1pt]{};
\coordinate[] (B5) at (3,0);
\node at (B5)[circle,fill,inner sep=1pt]{};
\coordinate[label =  above left: \text{\revision {\tiny $(n,0)$}}] (B6) at (4,0);
\coordinate[] (B6) at (4,0);
\node at (B6)[circle,fill,inner sep=1pt]{};

\draw[line width=0.2mm,-{Latex[length=1.5mm]}](B1) to [out=135,in=225] node[left]{\tiny $\fs+\nfs$} (B0);
\draw[line width=0.2mm,-{Latex[length=1.5mm]}](B1) to [out=-45,in=45] node[left]{\tiny $\nfd\fsprob+\fd + n\sr$} (B2);
\draw[line width=0.2mm,-{Latex[length=1.5mm]}](B1) to [out=0,in=90] (B3)node[below]{\tiny $\nfd(1-\fsprob)\fsprob$};
\draw[line width=0.2mm,-{Latex[length=1.5mm]}](B1) to [out=0,in=90] (B4)node[below]{\tiny $\nfd(1-\fsprob)^2\fsprob$};
\draw[line width=0.2mm,-{Latex[length=1.5mm]}](B1) to [out=0,in=90]  (B5)node[below]{\tiny $\nfd(1-\fsprob)^{n-1}\fsprob$};
\draw[line width=0.2mm,-{Latex[length=1.5mm]}](B1) to [out=0,in=90] (B6)node[below right]{\tiny $\nfd(1-\fsprob)^n$};

\coordinate[] (C0) at (7,6);
\node at (C0)[circle,fill,inner sep=1pt]{};
\coordinate[label =  above right : \text{\tiny $(m,n)$}] (C1) at (7,5);
\node at (C1)[circle,fill,inner sep=1pt]{};
\coordinate[] (C2) at (6,5);
\node at (C2)[circle,fill,inner sep=1pt]{};
\coordinate[] (C3) at (7,4);
\node at (C3)[circle,fill,inner sep=1pt]{};
\coordinate[] (C4) at (7.5,3);
\node at (C4)[circle,fill,inner sep=1pt]{};
\coordinate[] (C5) at (8,2);
\node at (C5)[circle,fill,inner sep=1pt]{};
\coordinate[] (C6) at (8.5,0);
\node at (C6)[circle,fill,inner sep=1pt]{};
\coordinate[label =  above right: \text{\revision{\tiny $(m+n,0)$}}] (C7) at (9.5,0);
\coordinate[] (C7) at (9.5,0);
\node at (C7)[circle,fill,inner sep=1pt]{};

\draw[line width=0.2mm,-{Latex[length=1.5mm]}](C1) to [out=135,in=225] node[left]{\tiny $\fs+\nfs$} (C0);
\draw[line width=0.2mm,-{Latex[length=1.5mm]}](C1) to [out=-135,in=-45] node[below]{\tiny $\fd+m\sr$} (C2);
\draw[line width=0.2mm,-{Latex[length=1.5mm]}](C1) to [out=-45,in=45]  (C3)node[below left ]{\tiny $\nfd\fsprob+n\sr$};

\draw[line width=0.2mm,-{Latex[length=1.5mm]}](C1) to [out=0,in=90] (C4)node[below left ]{\tiny $\nfd(1-\fsprob)\fsprob$};
\draw[line width=0.2mm,-{Latex[length=1.5mm]}](C1) to [out=0,in=90] (C5)node[below left ]{\tiny $\nfd(1-\fsprob)^2\fsprob$};
\draw[line width=0.2mm,-{Latex[length=1.5mm]}](C1) to [out=0,in=90]  (C6)node[below]{\tiny $\nfd(1-\fsprob)^{n-1}\fsprob$};
\draw[line width=0.2mm,-{Latex[length=1.5mm]}](C1) to [out=0,in=90] (C7)node[below right]{\tiny $\nfd(1-\fsprob)^n$};

\draw[-{Latex[length=1.5mm]},dashed,line width=0.05mm](0,0)--(11,0);
\node at (11.5,-0.25){ $m$};
\draw[-{Latex[length=1.5mm]},dashed,line width=0.05mm](0,0)--(0,7);
\node at (-0.25,7){$n$};

\draw[dashed,line width=0.05mm](0,5)--(10.5,5);
\draw[dashed,line width=0.05mm](7,0)--(7,6.5);

\draw[dashed,line width=0.05mm](0,1)--(1.5,1);
\draw[dashed,line width=0.05mm](1,0)--(1,1.5);

\end{tikzpicture}
\caption{Transition diagram for one-sided system.}
\label{fig:one-sided-state-diag}
\end{figure}

We use $\pi_{m,n}$ to denote the stationary probability of state $(m,n)\in \mathfrak{s}$. These probabilities must satisfy the following set of equilibrium equations. 
\fcu{There are four different cases corresponding to four different regions of the state space; we illustrate each case in Figure \ref{fig:one-sided-state-diag}.} For completeness, we present the steady-state equilibrium equations below. 

\begin{align}\label{eq:NB-st-1}
\nonumber
 \pi_{m,n}\left(\mu_1+\mu_2+\lambda_1+\lambda_2 + (m + n)\theta_s\right)
&=\pi_{m,n+1}(n+1)\theta_s +\pi_{m+1,n}(\mu_2+(m+1)\theta_s)+\pi_{m,n-1}(\lambda_1+\lambda_2)\\
&+\sum_{k=0}^m\pi_{m-k,n+k+1}\mu_1\gamma_s(1-\gamma_s)^k,\quad m\ge 1, n\ge 1;\\ \nonumber \\
\label{eq:NB-st-2} \nonumber 
\pi_{m,0}(\mu_2+\lambda_1+\lambda_2 + m\theta_s)&=\pi_{m,1}\theta_s + \pi_{m+1,0}\left(\mu_2 + \theta_s(m + 1)\right) \\
&+ \sum_{k=0}^m\pi_{m-k,k+1}\mu_1\gamma_s(1-\gamma_s)^k +\sum_{k=1}^{m}\pi_{m-k,k}\mu_1(1-\gamma_s)^k, \quad  m\ge1; \\ \nonumber \\
\nonumber
\pi_{0,n}(\mu_1+\mu_2+\lambda_1+\lambda_2  + n\theta_s) &=\pi_{0,n+1}\left((n+1)\theta_s + \mu_1\gamma_s + \mu_2\right) +\pi_{1,n}(\mu_2 + \theta_s)\\\label{eq:NB-st-3}
&+ \pi_{0,n-1}(\lambda_1+\lambda_2),\quad n\ge 1; \\ \nonumber \\ \label{eq:NB-st-4}
\pi_{0,0}(\lambda_1+\lambda_2 ) &=
\pi_{0,1}(\theta_s + \mu_1\gamma_s + \mu_2) + \pi_{1,0}(\mu_2 + \theta_s). \\ \nonumber
\end{align}

\subsection{The System without Reneging}

We start our analysis by first looking at the system without reneging ($\sr=0$). The proof of the following proposition is relegated to the appendix.

\begin{proposition}\label{prop:no-reneg}
Suppose there is no reneging in the N-system matching queue and that the following stability conditions are satisfied:
\begin{equation*}
    \lambda_1+\lambda_2<\mu_1+\mu_2\quad \text{and}\quad
    \
    \lambda_2<\mu_2.
\end{equation*}
Then the {\revision system is ergodic and its} steady state probabilities are given by 
\begin{equation*}
\pi_{m,n}=\left\{\begin{array}{ll}
\frac{\mu_1}{\mu_1+\mu_2}\left(\frac{\lambda_2}{\mu_2}\right)^m\left(\frac{\lambda_1+\lambda_2}{\mu_1+\mu_2}\right)^n B, & ~~ m\ge1; \\
\left(\frac{\lambda_1+\lambda_2}{\mu_1+\mu_2}\right)^nB, & ~~ m=0, \end{array} \right.
\end{equation*}
where $B$ is a normalizing constant which corresponds to $\pi_{0,0}$ --- the steady state probability of the system being empty. {\revision Here $B$ has an explicit form of $(\mu_1+\mu_2-\lambda_1-\lambda_2)(\mu_2-\lambda_2)/\left((\mu_1+\mu_2-\lambda_2)\mu_2\right)$.}
\end{proposition}

\medskip
We note that similar results have been established in previous literature. In the parallel FCFS system, \cite{adan2009exact} obtain a quite similar product form under a specific assignment probability when arriving jobs see idle servers; \cite{adan2014skill} also derive a related product form solution under the ALIS policy.
{\revision The stability conditions are crucial to make sure that the underlying Markov process is positive recurrent and failing the conditions intuitively leads to  growing queue lengths without bound over time. Note here that stability is possible only because the demand side arrivals does not queue. In a two-sided system as we will discuss in Section \ref{sec:two-sided}, no stability conditions exists (see e.g., Theorem 1 in \cite{mairesse2016stability} for an instability result of stochastic matching under bipartite graphs). Thus introducing reneging is one of the ways to stabilize the system.}
In the next section we make use of this result to provide some insights on the impact of reneging.

\subsection{The Impact of Reneging}\label{sec:imp-reneg}

\fcu{In this section we briefly discuss to what extent reneging impacts the analysis of a product form solution for our system and the main challenges that emerge.}

{\revision One immediate consequence of introducing reneging is that it stabilizes the system as long as reneging rate is positive. This is because the total reneging rate $\theta_s(m+n)$ will exceed any fixed supply arrival rate as queue lengths increase which makes the Markov process positive recurrent.} \fcu{We know that the steady state probabilities in Proposition \ref{prop:no-reneg} would not hold with reneging. The reason is that with reneging, the rate at which the system transitions depends on the number of supplies in the system.}
\fcu{One idea is to examine if the structure of the steady state probabilities for the classic M/M/1+M queue can somewhat be translated to our setting.} For this system it is well known that the steady state probabilities are of the form as $\prod_{i=1}^n \frac{\lambda}{\mu+i\theta}$ for some arrival and service rates $\lambda,\mu$ and reneging rate $\theta$. Given this, a natural way to extend the solution in Proposition
\ref{prop:no-reneg} to the setting with reneging is to consider a solution of the form 
\begin{equation}\label{eq:prod-fail}
\pi_{m,n}\propto \Big(\prod_{i=1}^{m} x_i\Big)\Big(\prod_{i=1}^{n}y_i\Big),
\end{equation}
where $x_i$ and $y_i$ 
must be found using equations \eqref{eq:NB-st-1} to \eqref{eq:NB-st-4}. \fcu{Here the different values of $x_i$ and $y_i$ account for the reneging rate at different levels of supply in the known and unknown parts of the queue.}
However, there is an immediate problem that arises when considering solutions given by \eqref{eq:prod-fail}. To see this, note that the main challenge in solving the system of equations \eqref{eq:NB-st-1} to 
\eqref{eq:NB-st-4} consists in dealing with the terms
\begin{equation}\label{eq:key-terms}
    \sum_{k=0}^m\pi_{m-k,n+k+1}(1-\gamma_s)^k
    \quad \text{and}\quad 
    \sum_{k=1}^{m}\pi_{m-k,k}(1-\gamma_s)^k.
\end{equation}
\fcu{At a first glance there is no clear way to simplify these two summations.} For example, 
if we plugin \eqref{eq:prod-fail} into the second term above we obtain the following non-trivial expression
\begin{equation*}
    \sum_{k=1}^{m}\pi_{m-k,k}(1-\gamma_s)^k
    \propto \sum_{k=1}^{m}\Big(\prod_{i=1}^{m-k} x_i\Big)\Big(\prod_{i=1}^{k}y_i\Big)(1-\gamma_s)^k.
\end{equation*}
In contrast, in the case without reneging we have
 $x_i=x=\frac{\lambda_2}{\mu_2}$ and $y_i=y=\frac{\lambda_1+\lambda_2}{\mu_1+\mu_2}$, therefore, 
 \begin{equation*}
     \pi_{m-k,k}\propto x^{m-k} y^k=\left(\frac{x}{y}\right)^{m-k}y^m,
 \end{equation*}
which can be used to show that
 \begin{align}
 \label{eq:no-reng-simpl}
     \sum_{k=1}^m\pi_{m-k,k}(1-\gamma_s)^k=ax^m,
 \end{align}
 for some positive constant $a$. The details of this simplification are presented in the proof of Proposition \ref{prop:no-reneg} in the appendix.
 This is a much more amenable expression that turns out to be fundamental in deriving Proposition \ref{prop:no-reneg}.
 Note that in general,  for non-constant $x_i$ and $y_i$ this simplification is not possible. However, in the next section we leverage some of these ideas to obtain a product form solution for the system with reneging.

 \subsection{The Product Form}

 We first state our main result which is the product form solution for the \fcu{N-system matching queue} with reneging. Then we comment on how this result compares to the case without reneging. After this we discuss the main ideas we use to develop our product form solution. 
\begin{theorem}\label{thm-prod-one-sided} The steady state probabilities for the one-sided N-system with reneging are given by 
\begin{equation*}
\pi_{m,n}=\left\{\begin{array}{ll}
\bigg(\frac{\mu_1}{\mu_1 + \mu_2 + m\theta_s }\bigg)\bigg(\prod_{i=1}^m\frac{\lambda_2}{\mu_2 + i\theta_s}\bigg)\bigg(\prod_{i=1}^n\frac{\lambda_1+\lambda_2}{\mu_1 + \mu_2 + m\theta_s + i\theta_s}\bigg)B, & ~~ m\ge1 \\
\bigg(\prod_{i=1}^n\frac{\lambda_1+\lambda_2}{\mu_1 + \mu_2 + i\theta_s}\bigg)B, & ~~ m=0 \end{array} \right.
\end{equation*}
where $B$ is a normalizing constant equal to $\pi_{0,0}$ --- the steady state probability of the system being empty.
\end{theorem}

{\revision Note that because of reneging, $B$ no longer has a simplified closed form. It is calculated as $B=1-\sum_{m+n\ge1}\pi_{m,n}$.} Recall that in Proposition \ref{prop:no-reneg} without reneging, the steady state probabilities for $m\ge 1$ are given by
\begin{equation*}
    \frac{\mu_1}{\mu_1+\mu_2}\left(\frac{\lambda_2}{\mu_2}\right)^m\left(\frac{\lambda_1+\lambda_2}{\mu_1+\mu_2}\right)^n,
\end{equation*}
 which have a clear resemblance with Theorem \ref{thm-prod-one-sided}. However, because of abandonment, the steady state probabilities now become a product of multiple terms that account for the reneging rate at different queue lengths for the known and unknown parts in the state description. Despite the resemblance, we note that our solution is not a simple extension of the form \eqref{eq:prod-fail},
 $\Big(\prod_{i=1}^{m} x_i\Big)\Big(\prod_{i=1}^{n}y_i\Big)$. There are at least two main differences. 
First, the term $\mu_1/(\mu_1 + \mu_2 + m\theta_s)$ is not accounted for in \eqref{eq:prod-fail}. Second, and more importantly, in our expression $y_i$ depends on $m$ --- it equals 
$(\lambda_1+\lambda_2)/(\mu_1 + \mu_2 + m\theta_s + i\theta_s)$ --- and not solely on $i$. The intuition for this is that the known part of the queue (the $m$ part) impacts the evolution of the unknown part of the queue (the $n$ part) by increasing the rate at which the unknown supply become inflexible supply. In turn, including reneging to the system leads to
a non-trivial product form solution in which the interaction between known and unknown parts in the state description must be taken into account.

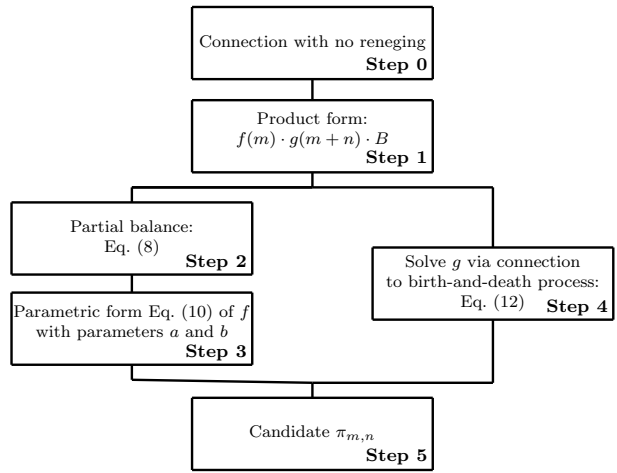
\begin{figure}[h]
\centering
\scalebox{0.8}{\begin{tikzpicture}[baseline=0pt]
\def\sy{0.6};
\def\sx{2};
\coordinate (A) at (5,5.25);
\node at (A) {Connection with no reneging};
\draw[-,line width=0.4mm]($(A)+(-\sx,\sy)$)--($(A)+(\sx,\sy)$);
\draw[-,line width=0.4mm]($(A)+(-\sx,-\sy)$)--($(A)+(\sx,-\sy)$);
\draw[-,line width=0.4mm]($(A)+(-\sx,-\sy)$)--($(A)+(-\sx,\sy)$);
\draw[-,line width=0.4mm]($(A)+(\sx,\sy)$)--($(A)+(\sx,-\sy)$);

\coordinate (B) at (5,3.7);
\node[align=center] at ($(B)+(0,+0.1)$) {Product form:\\ $f(m)\cdot g(m+n)\cdot B$};
\draw[-,line width=0.4mm]($(B)+(-\sx,\sy)$)--($(B)+(\sx,\sy)$);
\draw[-,line width=0.4mm]($(B)+(-\sx,-\sy)$)--($(B)+(\sx,-\sy)$);
\draw[-,line width=0.4mm]($(B)+(-\sx,-\sy)$)--($(B)+(-\sx,\sy)$);
\draw[-,line width=0.4mm]($(B)+(\sx,\sy)$)--($(B)+(\sx,-\sy)$);

\coordinate (C) at (2,2.0);
\node[align=center] at (C) {Partial balance:\\ Eq. \eqref{eq:key-step-partial}};
\draw[-,line width=0.4mm]($(C)+(-\sx,\sy)$)--($(C)+(\sx,\sy)$);
\draw[-,line width=0.4mm]($(C)+(-\sx,-\sy)$)--($(C)+(\sx,-\sy)$);
\draw[-,line width=0.4mm]($(C)+(-\sx,-\sy)$)--($(C)+(-\sx,\sy)$);
\draw[-,line width=0.4mm]($(C)+(\sx,\sy)$)--($(C)+(\sx,-\sy)$);

\coordinate (D) at (2,0.5);
\node[align=center] at ($(D)+(0,0.1)$) {Parametric form Eq. \eqref{eq:sol-f} of $f$ \\ with parameters $a$ and $b$};
\draw[-,line width=0.4mm]($(D)+(-\sx,\sy)$)--($(D)+(\sx,\sy)$);
\draw[-,line width=0.4mm]($(D)+(-\sx,-\sy)$)--($(D)+(\sx,-\sy)$);
\draw[-,line width=0.4mm]($(D)+(-\sx,-\sy)$)--($(D)+(-\sx,\sy)$);
\draw[-,line width=0.4mm]($(D)+(\sx,\sy)$)--($(D)+(\sx,-\sy)$);

\coordinate (E) at (8,1.25);
\node[align=center] at (E) {Solve $g$ via \revision{connection}\\
\revision{to birth-and-death process}:\\ Eq. \eqref{eq:recur-g} };
\draw[-,line width=0.4mm]($(E)+(-\sx,\sy)$)--($(E)+(\sx,\sy)$);
\draw[-,line width=0.4mm]($(E)+(-\sx,-\sy)$)--($(E)+(\sx,-\sy)$);
\draw[-,line width=0.4mm]($(E)+(-\sx,-\sy)$)--($(E)+(-\sx,\sy)$);
\draw[-,line width=0.4mm]($(E)+(\sx,\sy)$)--($(E)+(\sx,-\sy)$);

\coordinate (G) at (5,-1.25);
\node[align=center] at (G){Candidate $\pi_{m,n}$};
\draw[-,line width=0.4mm]($(G)+(-\sx,\sy)$)--($(G)+(\sx,\sy)$);
\draw[-,line width=0.4mm]($(G)+(-\sx,-\sy)$)--($(G)+(\sx,-\sy)$);
\draw[-,line width=0.4mm]($(G)+(-\sx,-\sy)$)--($(G)+(-\sx,\sy)$);
\draw[-,line width=0.4mm]($(G)+(\sx,\sy)$)--($(G)+(\sx,-\sy)$);


\draw[-,line width=0.4mm]($(A)+(0,-\sy)$)--($(B)+(0,\sy)$);
\draw[-,line width=0.4mm]($(C)+(0,-\sy)$)--($(D)+(0,\sy)$);
\draw[-,line width=0.4mm]($(E)+(0,-\sy)$)--($(E)+(0,\sy-2.25)$);
\draw[-,line width=0.4mm]($(E)+(0,\sy-2.25)$)--($(G)+(0,\sy+0.25)$);
\draw[-,line width=0.4mm]($(E)+(0,\sy)$)--($(E)+(0,\sy+1.0)$);
\draw[-,line width=0.4mm]($(E)+(0,\sy+1.0)$)--($(B)+(0,-\sy-0.25)$);
\draw[-,line width=0.4mm]($(B)+(0,-\sy)$)--($(B)+(0,-\sy-0.25)$);
\draw[-,line width=0.4mm]($(C)+(0,\sy+0.25)$)--($(B)+(0,-\sy-0.25)$);
\draw[-,line width=0.4mm]($(C)+(0,\sy+0.25)$)--($(C)+(0,\sy)$);
\draw[-,line width=0.4mm]($(C)+(0,\sy+0.25)$)--($(B)+(0,-\sy-0.25)$);

\draw[-,line width=0.4mm]($(G)+(0,\sy)$)--($(G)+(0,\sy+0.25)$);
\draw[-,line width=0.4mm]($(D)+(0,-\sy-0.25)$)--($(G)+(0,\sy+0.25)$);
\draw[-,line width=0.4mm]($(D)+(0,-\sy-0.25)$)--($(D)+(0,-\sy)$);

\node[rotate=0,font=\fontsize{8.5}{3.5}\selectfont] at ($(A)+(1.4,-0.375)$){\textbf{Step 0}};
\node[rotate=0,font=\fontsize{8.5}{3.5}\selectfont] at ($(B)+(1.4,-0.375)$){\textbf{Step 1}};
\node[rotate=0,font=\fontsize{8.5}{3.5}\selectfont] at ($(C)+(1.4,-0.4)$){\textbf{Step 2}};
\node[rotate=0,font=\fontsize{8.5}{3.5}\selectfont] at ($(D)+(1.4,-0.4)$){\textbf{Step 3}};
\node[rotate=0,font=\fontsize{8.5}{3.5}\selectfont] at ($(E)+(1.4,-0.4)$){\textbf{Step 4}};
\node[rotate=0,font=\fontsize{8.5}{3.5}\selectfont] at ($(G)+(1.4,-0.375)$){\textbf{Step 5}};

\end{tikzpicture}}
\caption{Key ideas for developing the product form. }
\label{fig:key-ideas}
\end{figure}

We now highlight the main ideas that lead us to Theorem \ref{thm-prod-one-sided}, see Figure \ref{fig:key-ideas} for a schematic representation. 

 \textbf{Step 0.} As  noted in Section \ref{sec:imp-reneg}, in order to obtain the steady state probabilities, the key terms to analyze are $\pi_{m-k,n+k+1}$ and $\pi_{m-k,k}$, see Eq. \eqref{eq:key-terms}. With this in mind, let us momentarily consider the setting without reneging
 and let $\tilde{\pi}_{m,n}$ denote its steady state probability.
In that setting the following transformation is useful in simplifying terms in Eq. \eqref{eq:key-terms}:
\begin{equation*}
\tilde{\pi}_{m-k,k}\propto x^{m-k} y^k= (x/y)^{m-k} y^{(m-k)+k},
\end{equation*}
where $x,y$ are given by Proposition \ref{prop:no-reneg}. From this we can observe that  $\tilde{\pi}_{m-k,k}$ is a function of its first component, $m-k$, and the sum of its two components, $m$.
 More generally, by rearranging terms in Proposition \ref{prop:no-reneg}, it is possible to see that the steady state probabilities for any state $(m,n)$ are a function of $m$ and $m+n$. We can now look at the case with reneging through this lens. 
 
\textbf{Step 1.} Suppose that there is reneging in the system. If we allow for a similar form as in the previous step in which $\pi_{m,n}$ depends on $m$ and $m+n$, we can simplify the expression in Eq. \eqref{eq:key-terms}. That is, assuming that 
$$\pi_{m,n}=f(m)g(m+n)\cdot B,~\forall m,n\ge0,$$ 
for some functions $f$ and $g$ such that $f(0)=g(0)=1$ to be determined, where $B$ is a normalizing constant. 
Then 
 \begin{equation*}
 \sum_{k=1}^{m}\pi_{m-k,k}(1-\gamma_s)^k=
 g(m)\sum_{k=1}^{m}f(m-k)(1-\gamma_s)^k.
 \end{equation*}
 
 \textbf{Step 2.} The next step is to consider a key partial balance condition which lead us to a simplified form for the summation $\sum_{k=1}^{m}f(m-k)(1-\gamma_s)^k$. To do that, we draw inspiration from the case without reneging. Recall that from Eq. \eqref{eq:no-reng-simpl} we have
 \begin{equation*}
     \sum_{k=1}^{m}\tilde{\pi}_{m-k,k}(1-\gamma_s)^k= a x^m=\tilde{\pi}_{m,0}\frac{\mu_2}{\mu_1}.
 \end{equation*}
 Multiplying both sides above by $\mu_1$ we observe the following 
 key partial balance property is satisfied: fix state $(m,0)$, then, in steady state, the flow into this state due to all unknown supply becoming inflexible supply equals the flow out of $(m,0)$
 due to inflexible supply leaving the system. 
 If we translate this property to the context with reneging 
 we would have 
  \begin{equation}\label{eq:key-step-partial}
     \sum_{k=1}^{m}\pi_{m-k,k}\mu_1 (1-\gamma_s)^k =\pi_{m,0}\cdot (\mu_2+\sr\cdot m).
 \end{equation}
This is similar to the case without reneging but we have to adjust the rate at which inflexible supply leaves state $(m,0)$. In the case with reneging inflexible supply leaves $(m,0)$ at rate $(\mu_2+\sr\cdot m)$.
In turn, the equation above establishes the partial balance property for the setting with abandonment. In Figure \ref{fig:partial_balance} we provide a graphical illustration of this property.

We can then replace $\pi_{m-k,k}$ and $\pi_{m,0}$ in the partial balance condition above by using $f$ and $g$. This delivers
 \begin{equation}\label{eq:key-step}
 \sum_{k=1}^{m}f(m-k)(1-\gamma_s)^k=f(m)\cdot (a+m\cdot b),
 \end{equation}
 where 
  \begin{equation*}
 a= \frac{\mu_2}{\mu_1}\quad \text{and}\quad b=\frac{\theta_s}{\mu_1}.
 \end{equation*}

We refer to Eq. \eqref{eq:key-step-partial} as the partial balance condition. 

\begin{figure}
    \centering
    \includegraphics[width=0.45\textwidth]{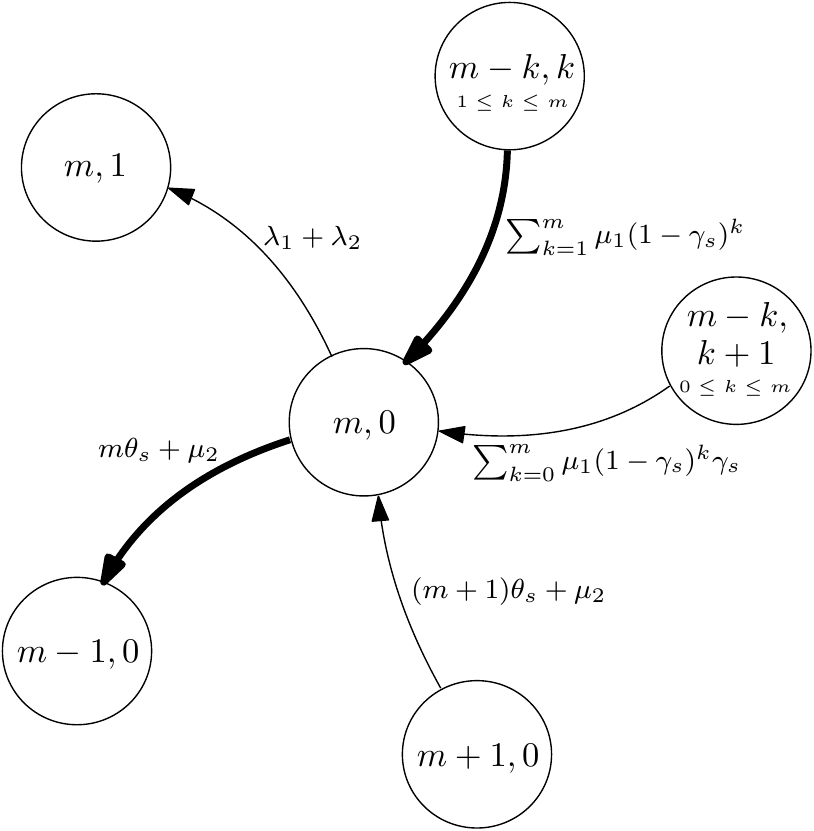}
    \caption{Partial balance at state $(m,0)$: the flow from states $(m-k,k)$ into $(m,0)$ equals the flow out of $(m,0)$ to state $(m-1,0)$; the bold arrows in the figure indicate the transition rates involved in the partial balance equation.}
    \label{fig:partial_balance}
\end{figure}

\textbf{Step 3.} From the partial balance condition in the previous step we can completely solve for $f$. An inductive argument shows that 
 any function $f$ with $f(0)=1$  that satisfies Eq. \eqref{eq:key-step} is of the form
 \begin{equation}\label{eq:sol-f}
 f(m)=\frac{(1-\gamma_s)^m}{a+b}\left(\prod_{i=2}^m \frac{1+a+(i-1)b}{a+ib}\right),\quad \forall m\geq1,
 \end{equation}
  we prove this result in Lemma \ref{lm:key-step-induction} in the appendix. In turn, given our product form solution, 
 $\pi_{m,n}=f(m)g(m+n)\cdot B$, we have characterized the steady state probabilities up to $g$. In the next step we discuss how to solve for $g$.
 
 \textbf{Step 4.} 
 We use
 Eq. \eqref{eq:NB-st-4} and the formula for $f(1)$ in Eq. \eqref{eq:sol-f} to solve for $g(1)$ as a function
 of both $a$ and $b$. We obtain
 \begin{equation}
 \label{eq:g-1}
 g(1)=\frac{\lambda_1+\lambda_2}{(\theta_s + \mu_1\gamma_s + \mu_2) +\frac{(1-\gamma_s)}{a+b}(\mu_2 + \theta_s)}.
 \end{equation}
 Then, we can employ Eq. \eqref{eq:NB-st-3} to solve for $g(m)$ for $m\geq 2$,
 \begin{equation}\label{eq:recur-g}
 g(m) =\frac{g(m-1)(\mu_1+\mu_2+\lambda_1+\lambda_2  + (m-1)\theta_s)-g(m-2)(\lambda_1+\lambda_2)}{\left(m\theta_s + \mu_1\gamma_s + \mu_2\right)+\frac{(1-\gamma_s)}{a+b}(\mu_2 + \theta_s)}.
 \end{equation}
 Note that this recursion is completely parametrized by $a$ and $b$. As a consequence, by solving this recursion we completely characterize the steady state probabilities. 
 {\revision Interestingly, by replacing the values of $a$ and $b$, this set of equations coincides exactly with the balance equations for a birth-and-death process with birth rates $\lambda_1+\lambda_2$ and  death rates 
 $\mu_1+\mu_2 +m\theta_s$:
 \begin{equation*}
     g(m-1)(\mu_1+\mu_2+\lambda_1+\lambda_2  + (m-1)\theta_s)=
     g(m)(\mu_1+\mu_2+m\theta_s)+g(m-2)(\lambda_1+\lambda_2),\quad m\ge 2,
 \end{equation*}
 with $g(1) = (\lambda_1+\lambda_2)/(\mu_1+\mu_2+\theta_s)$. Hence, the solution of $g(\cdot)$ is given by 
 }

 \begin{equation*}
 g(m)=\left(\prod_{i=1}^{m}\frac{\lambda_1+\lambda_2}{\mu_1+\mu_2+i \theta_s}\right) .
 \end{equation*}
  
 \textbf{Step 5.} We are now ready to obtain the product form. Replacing the values of $a$ and $b$ in Eq. \eqref{eq:sol-f}, and 
 using the product form $f(m)\cdot g(m+n)$ we obtain
 \begin{equation*}
     \pi_{m,n} = \frac{\mu_1(1-\gamma_s)^m}{\mu_2+\theta_2}\left(\prod_{i=2}^{m}\frac{(\mu_1+\mu_2+(i-1)\theta_s)}{\mu_2+i \theta_s}\right)\left(\prod_{i=1}^{m+n}\frac{\lambda_1+\lambda_2}{\mu_1+\mu_2+i \theta_s}\right)\cdot B, \quad \forall m\geq 1
 \end{equation*}
 For $m=0$, because $f(0)=1, g(m+n)=g(n)$, we have $\pi_{0,n}=\left(\prod_{i=1}^{n}\frac{\lambda_1+\lambda_2}{\mu_1+\mu_2+i \theta_s}\right)\cdot B$. Hence, $\pi_{m,n}$ is fully determined up to a normalizing constant. 
 
 Several comments are in order. In Theorem
 \ref{thm-prod-one-sided} we have presented the steady state probabilities in a simplified, more natural, form that is different from the equation above. In the proof of Theorem 
 \ref{thm-prod-one-sided} we prove that these expressions are equivalent. Also note that in the series of steps outlined above we have not verified all the equilibrium equations, Eq. \eqref{eq:NB-st-1} to Eq.
 \eqref{eq:NB-st-4}.

 Nevertheless, since $g$ solves the recursion in Eq. \eqref{eq:recur-g} it follows that 
  Eq. \eqref{eq:NB-st-3} and Eq. \eqref{eq:NB-st-4} are satisfied. 
  So that the only set of equations that remain to be checked are 
  Eq. \eqref{eq:NB-st-2}  and 
  Eq. \eqref{eq:NB-st-1}. In the theorem's proof we show that these equations are indeed verified by our solution. 
  Finally, we note that there are two key ideas in the previous derivations. First, we impose a product form solution in which $\pi_{m,n}$ equals the product of two functions: one that depends on the number of known supply in the system, $f(m)$, and another that depends on the total number of agents in the system, $g(m+n)$. Second, these functions can be fully determined by an appropriate partial balance condition, see Eq. \eqref{eq:key-step-partial} and Figure \ref{fig:partial_balance}.

\section{Two-sided System}
\label{sec:two-sided}
Consider now an extension of the system analyzed in Section \ref{sec:one-sided} in which demand units can queue. We refer to this new system as
the two-sided queue. Specifically, we assume flexible and inflexible supply (demand) arrive as independent Poisson streams, with rate $\fs$ ($\fd$) and $\nfs$ ($\nfd$). Flexible supply can be matched to either type 1 or type 2 demand, and inflexible supply can be matched only to type 2 demand in an FCFS manner. Supply (demand) stays in the system for an i.i.d. and exponentially distributed time with rate $\sr$ ($\dr$). Supply and demand can have different reneging rates. In what follows, to streamline exposition, we refer to the supply and demand side as the left and right side interchangeably, see Figure \ref{fig:n-system-intro}.

\subsection{State Space Representation}
\label{subsec:state_space_representation-two-sided}

There are three possible scenarios that can emerge. Supply can be queuing on the left side but there is no demand queuing on the right side. In this case we are in the exact same situation as in Section \ref{sec:one-sided}. However, when supply units begin to deplete, it is possible that demand units will have to wait for a compatible supply. Note that if there is  type 2 demand waiting then it must be that there is no supply waiting; but if there is type 1 demand waiting then it is possible that some inflexible supply is queuing. This gives rise to two additional system configurations. 
In one of them demand can be queuing on the right side and there is no supply on the left side. Importantly, this scenario is symmetric to the one-sided case studied in the previous section. In the other, there could be supply and demand queuing on both sides at the same time. The latter occurs only when supply and demand units are not compatible, that is, supplies are inflexible and demands are of type 1. This discussion motivates us to split the state description into three cases (which we depict in Figure \ref{fig:3-cases} below):
\begin{itemize}
    \item[(a)] \textbf{Queue on left side.} There is some supply queuing on the left side and exactly zero demand queuing on the right side. We use the same state representation, $(m,n)$, as in the previous section where  $m$ is the number of \emph{known} inflexible supplies at the beginning of the queue, and $n$ is the remaining number of supplies with \emph{unknown} type, i.e. they could be either flexible or inflexible. In this case we consider $(m,n)$ in $\mathfrak{s}=\left\{(m,n) \mid m,n\in\mathbb{Z}_{\ge0}, m+n>0 \right\}$, that is, there is always someone queuing on the left side (we treat the case $m=0,n=0$ separately). 
    See Figure \ref{fig:3-cases}(a).
    \item[(b)] \textbf{Queue on right side.} There is some demand queuing on the right side and exactly zero supply queuing on the left side. We use $(m,n)$ to represent the state of the system where $m$ is the number of \emph{known} type 1 demand at the beginning of the queue, and $n$ is the remaining number of demand with \emph{unknown} type. In this case we consider $(m,n)$ in $\mathfrak{s}$, that is, there is always someone queuing on the right side.
    Similarly to the previous section we denote $\fdprob=\mu_2/(\mu_1 + \mu_2)$ as the probability of a demand arrival being type 2.
    See Figure \ref{fig:3-cases}(b).
    \item[(c)] \textbf{Queue on both sides.} There is some inflexible supply queuing on the left side and some type 1 demand queuing on the right side. We use $(i,j)$ to denote the state of the system with $i,j\ge1$ where $i$ is the number of inflexible supplies on the left queue and $j$ is the number of type 1 demands on the right queue. Note that the system can be in this state only when there is a positive number of both inflexible supply and type 1 demand (and no inflexible supply and type 2 demand). In this state there are no feasible matches among the agents in the system. See Figure \ref{fig:3-cases}(c).
\end{itemize}
\begin{figure}[h]
\centering
\scalebox{0.6}{\begin{tikzpicture}[baseline=0pt]
\def\sf{-2}
\def\sg{4.5}

\draw[line width=0.6mm] (-7,4.5)--(-4,4.5);
\draw[line width=0.6mm] (-7,3.5)--(-4,3.5);
\draw[line width=0.6mm] (-4,3.5)--(-4,4.5);

\draw[line width=0.6mm] (7-\sg,4.5)--(4-\sg,4.5);
\draw[line width=0.6mm] (7-\sg,3.5)--(4-\sg,3.5);
\draw[line width=0.6mm] (4-\sg,3.5)--(4-\sg,4.5);

\node at (-4.65,4.0) {\Large $m$};
\node at (-6,4.0) {\Large $n$};
\draw[line width=0.6mm] (-5.25,4.5)--(-5.25,3.5);
\node at (1.0,4.0) {\Large no demand};

\node at (-4.65,4.75){\large Type 2};
\node at (-6.2,4.75){\large Unknown};


\draw[line width=0.6mm] (7-\sg,4.5+\sf)--(4-\sg,4.5+\sf);
\draw[line width=0.6mm] (7-\sg,3.5+\sf)--(4-\sg,3.5+\sf);
\draw[line width=0.6mm] (4-\sg,3.5+\sf)--(4-\sg,4.5+\sf);

\draw[line width=0.6mm] (-7,4.5+\sf)--(-4,4.5+\sf);
\draw[line width=0.6mm] (-7,3.5+\sf)--(-4,3.5+\sf);
\draw[line width=0.6mm] (-4,3.5+\sf)--(-4,4.5+\sf);

\node at (-5.5,4.0+\sf) {\Large no supply}; 
\node at (0.15,4.0+\sf) {\Large $m$};
\node at (1.5,4.0+\sf) {\Large $n$};
\draw[line width=0.6mm] (0.75,4.5+\sf)--(0.75,3.5+\sf);

\node at (0.15,4.75+\sf){\large Type 1};
\node at (1.65,4.75+\sf){\large Unknown};

\draw[line width=0.6mm] (-7,4.5+2*\sf)--(-4,4.5+2*\sf);
\draw[line width=0.6mm] (-7,3.5+2*\sf)--(-4,3.5+2*\sf);
\draw[line width=0.6mm] (-4,3.5+2*\sf)--(-4,4.5+2*\sf);

\draw[line width=0.6mm] (7-\sg,4.5+2*\sf)--(4-\sg,4.5+2*\sf);
\draw[line width=0.6mm] (7-\sg,3.5+2*\sf)--(4-\sg,3.5+2*\sf);
\draw[line width=0.6mm] (4-\sg,3.5+2*\sf)--(4-\sg,4.5+2*\sf);

\node at (-5.45,4.0+2*\sf) {\Large $i$};
\node at (0.95,4.0+2*\sf) {\Large $j$};
\node at (0.95,4.75+2*\sf){\large Type 1};
\node at (-5.45,4.75+2*\sf){\large Type 2};

\node at (-8.5,4.0){\Large \textbf{Case (a):}};
\node at (-8.5,4.0+1*\sf){\Large \textbf{Case (b):}};
\node at (-8.5,4.0+2*\sf){\Large \textbf{Case (c):}};

\node at (6,5.0){\Large \textbf{Steady state probability}};
\node at (6,4.0){\Large $\pi^L_{m,n}$};
\node at (6,4.0+1*\sf){\Large $\pi^R_{m,n}$};
\node at (6,4.0+2*\sf){\Large $q_{i,j}$};
\end{tikzpicture}}
\caption{Two sided queue state space representation.}
\label{fig:3-cases}
\end{figure}
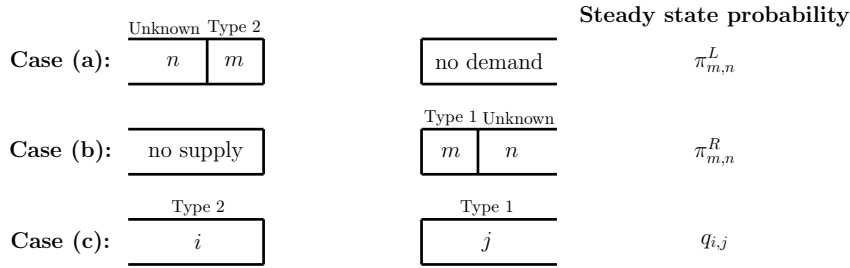

\subsection{Equilibrium Equations}
We use $\{\pi^L_{m,n}\}_{(m,n)\in\mathfrak{s} }$, $\{\pi^R_{m,n}\}_{(m,n)\in\mathfrak{s}}$ and $\{q_{i,j}\}_{i,j\ge 1}$ to denote the steady state probabilities for cases (a), (b) and (c), respectively. We use $\pi_{0,0}$ to denote the steady state probability of the empty system. These probabilities must satisfy the following set of equilibrium equations.

\begin{itemize}
\item[(a)] \textbf{Queue on left side:}
\begin{align}\nonumber
\pi^L_{m,n}(\mu_1+\mu_2+\lambda_1+\lambda_2 + (m + n)\sr) &=\pi^L_{m,n+1}  (n+1)\sr+\pi^L_{m+1,n}(\mu_2+(m+1)\sr)\\\nonumber
&+\pi^L_{m,n-1}(\lambda_1+\lambda_2)\\\label{eq:NB2-st-1-L}
&+\sum_{k=0}^m\pi^L_{m-k,n+k+1}\mu_1\fsprob(1-\fsprob)^k, \quad m\ge 1,n\geq 1; \\
\nonumber \\ \nonumber
\pi^L_{m,0}(\mu_1+\mu_2+\lambda_1+\lambda_2 + m\sr ) &= \pi^L_{m,1}\sr
+ \pi^L_{m+1,0}(\mu_2 + \sr(m + 1))\\\label{eq:NB2-st-2-L}
&+ q_{m,1}(\lambda_1 +\dr)+\sum_{k=0}^m\pi^L_{m-k,k+1}\mu_1\fsprob(1-\fsprob)^k,\quad m\geq 1; \\ \nonumber \\
\nonumber
\pi^L_{0,n}(\mu_1+\mu_2+\lambda_1+\lambda_2  + n\sr) &=
\pi^L_{0,n+1}((n+1)\sr + \mu_1\fsprob + \mu_2)+\pi^L_{1,n}(\mu_2 + \sr) \\\label{eq:NB2-st-3-L}
& + \pi^L_{0,n-1}(\lambda_1+\lambda_2),
\quad n\geq 1.
\end{align}
Note that equations from \eqref{eq:NB2-st-1-L} to \eqref{eq:NB2-st-3-L}
are similar to the equations that define the one-sided steady state probabilities. However they differ in that now the system can transition to case (a) from case (c) as well.
For example, consider equation \eqref{eq:NB2-st-2-L}. There the system can transition to state $(m,0)$ --- in which there is $m$ inflexible supply on the left side and no demand queuing on the right side --- from  state $(m,1)$ --- in which there is $m$ supplies queuing on the left side and one demand queuing on the right side (see the term $q_{m,1}$ in Eq. \eqref{eq:NB2-st-2-L}). To achieve this transition it is enough for the demand on the right side to renege or, also, it is possible that an arrival of flexible supply is matched to the demand.\\

\item[(b)] \textbf{Queue on right side:}
\begin{align}\nonumber
\pi^R_{m,n}(\mu_1+\mu_2+\lambda_1+\lambda_2 + (m + n)\dr )&=\pi^R_{m,n+1}  (n+1)\dr+\pi^R_{m+1,n}(\lambda_1+(m+1)\dr)\\\nonumber
&+\pi^R_{m,n-1}(\mu_1+\mu_2)
\\\label{eq:NB2-st-1-R}
&+\sum_{k=0}^m\pi^R_{m-k,n+k+1}\lambda_2\fdprob(1-\fdprob)^k, \quad m\ge 1,n\geq 1;
\\ \nonumber \\ \nonumber
\pi^R_{m,0}(\mu_1+\mu_2+\lambda_1+\lambda_2 + m\dr ) &= \pi^R_{m,1}\dr
+ \pi^R_{m+1,0}(\lambda_1 + (m + 1)\dr)\\ \label{eq:NB2-st-2-R}
&+q_{m,1}(\mu_2 + \sr)
+\sum_{k=0}^m\pi^R_{m-k,k+1}\lambda_2\fdprob(1-\fdprob)^k
,\quad m\geq1;
\\ \nonumber \\ \nonumber
\pi^R_{0,n}(\mu_1+\mu_2+\lambda_1+\lambda_2  + n\dr) &=
\pi^R_{0,n+1}((n+1)\dr + \lambda_2\fdprob + \lambda_1)\\\label{eq:NB2-st-3-R}
&+\pi^R_{1,n}(\lambda_1 + \dr)  + \pi^R_{0,n-1}(\mu_1+\mu_2),\quad n\geq 1.
\end{align}

This set of equations are symmetric to the previous one.\\

\item[(c)] \textbf{Queue on both sides:}
\begin{align}\nonumber
q_{i,j}(i\sr + j\dr + \mu_1 + \mu_2 + \lambda_1 + \lambda_2 )&= q_{i+1,j}((i+1)\sr + \mu_2)\\\label{eq:NB2-st-1-M}
&+ q_{i-1,j}\lambda_2 + q_{i,j+1}((j+1)\dr + \lambda_1) + q_{i,j-1}\mu_1,\quad i\ge 2, j\ge 2; \\ \nonumber \\ \nonumber
q_{1,j}(\sr + j\dr+ \mu_1 + \mu_2 + \lambda_1 + \lambda_2) &= q_{2,j}(2\sr + \mu_2)+ q_{1,j+1}((j+1)\dr + \lambda_1)\\\label{eq:NB2-st-2-M}
&+ q_{1,j-1}\mu_1
+\sum_{k=0}^j\pi^R_{j-k,k}\lambda_2(1-\fdprob)^k,\quad
j\geq 2;\\ \nonumber \\ \nonumber
q_{i,1}(i\sr + \dr + \mu_1 + \mu_2 + \lambda_1 + \lambda_2 ) &= q_{i+1,1}((i+1)\sr + \mu_2)
+ q_{i-1,1}\lambda_2 + q_{i,2}(2\dr + \lambda_1)\\\label{eq:NB2-st-3-M}
&+ \sum_{k=0}^i \pi^L_{i-k,k}\mu_1(1-\fsprob)^k,\quad 
i\ge 2; \\ \nonumber \\ \nonumber
q_{1,1}(\sr + \dr + \mu_1 + \mu_2 + \lambda_1 + \lambda_2 )&=q_{2,1}(2\sr + \mu_2)
+ \pi^R_{1,0}\lambda_2
+\pi^R_{0,1}\lambda_2(1-\fdprob)
+ q_{1,2}(2\dr + \lambda_1)\\\label{eq:NB2-st-4-M} 
&+ \pi^L_{1,0}\mu_1 + \pi^L_{0,1}\mu_1(1-\fsprob)
\end{align}
\end{itemize}

Eq. \eqref{eq:NB2-st-1-M} corresponds to a set of equilibrium equations when the inflexible supply and type 1 demand queues behave as two separate M/M/1+M queues. Eq. \eqref{eq:NB2-st-2-M} and \eqref{eq:NB2-st-3-M} contains an additional edge case of transition from case (a) or (b) to (c): the arriving type 1 (inflexible) demand (supply) into an empty demand (supply) queue finds no compatible match and queues in the system. Similarly, Eq. \eqref{eq:NB2-st-4-M} contains two other edge scenarios at state $(1,1)$ where both cases (a) and (b) can transition to (c). \\

Finally, the following equation couples $\pi^R$, $\pi^L$ and $\pi_{0,0}$,
\begin{flalign}\label{eq:NB2-st-4-L}
\pi_{0,0}(\lambda_1+\lambda_2+\mu_1+\mu_2) =
\pi^L_{0,1}(\sr + \mu_1\fsprob + \mu_2) + 
\pi^L_{1,0}(\mu_2 + \sr)
+\pi^R_{1,0}(\dr+ \lambda_1)+
\pi^R_{0,1}(\dr + \lambda_1 + \lambda_2\fdprob).
\end{flalign}

\subsection{The Product Form}
Thanks to the intuition developed in Section \ref{sec:one-sided} and Theorem \ref{thm-prod-one-sided},
we can make an educated guess for the steady state probabilities of the two-sided matching queue with reneging. When the system is as in Figure \ref{fig:3-cases}(a), it operates in  the same fashion as the one-sided system of Section \ref{sec:one-sided}. In turn, the steady-state probabilities should be proportional to those presented in Theorem \ref{thm-prod-one-sided}. When the system is as in
Figure \ref{fig:3-cases}(b), the situation is symmetric, hence the steady state probabilities correspond to the mirror image of the ones for case (a). 
The situation is different for the case depicted in Figure \ref{fig:3-cases}(c). If we consider the left side of the system, the queue increases whenever there is an arrival of inflexible supply; and decreases whenever a supply reneges or there is an arrival of type 2 demand. Note that the arrival of flexible supply does not affect the queue on the left side because it is immediately matched with a unit of demand on the right side, thereby reducing the amount of type 1 demand. Also, an arrival of type 1 demand will not affect the supply queue but will increase the demand queue. As a consequence, in this case, both sides of the system behave like an M/M/1+M queue. The following theorem formalizes this discussion. 

\begin{theorem}\label{thm:two-sided}The steady state probabilities for the two sided N-system with reneging are given by 
\begin{align}
\label{eq:two_sided_mm1}
&q_{i,j} = \left(\prod_{k=1}^i\frac{\lambda_2}{\mu_2 + k\sr}\right)\left(\prod_{k=1}^j\frac{\mu_1}{\lambda_1 + k\dr}\right)B,\quad i,j\geq 1,\\ \label{eq:two_sided_left}
&\pi^L_{m,n}=\left\{\begin{array}{lr}
\bigg(\frac{\mu_1}{\mu_1 + \mu_2 + m\theta_s }\bigg)\bigg(\prod_{i=1}^m\frac{\lambda_2}{\mu_2 + i\theta_s}\bigg)\bigg(\prod_{i=1}^n\frac{\lambda_1+\lambda_2}{\mu_1 + \mu_2 + m\theta_s + i\theta_s}\bigg)B, & \hspace{-6mm}m\ge1 \\
\bigg(\prod_{i=1}^n\frac{\lambda_1+\lambda_2}{\mu_1 + \mu_2 + i\theta_s}\bigg)B, & \hspace{-6mm}m=0, n>0 \end{array} \right.\\ \label{eq:two_sided_right}
&\pi^R_{m,n}=\left\{\begin{array}{lr}
\bigg(\frac{\lambda_2}{\lambda_1 + \lambda_2 + m\theta_d }\bigg)\bigg(\prod_{i=1}^m\frac{\mu_1}{\lambda_1 + i\theta_d}\bigg)\bigg(\prod_{i=1}^n\frac{\mu_1+\mu_2}{\lambda_1 + \lambda_2 + m\theta_d + i\theta_d}\bigg)B, & \hspace{-6mm}m\ge1 \\
\bigg(\prod_{i=1}^n\frac{\mu_1+\mu_2}{\lambda_1 + \lambda_2 + i\theta_d}\bigg)B, & \hspace{-6mm}m=0, n>0 \end{array} \right.
\end{align}
where $B$ is a normalizing constant equal to $\pi_{0,0}$ --- the steady state probability of having an empty system. 
\end{theorem}

{\revision Similar to Theorem \ref{thm-prod-one-sided}, $B$ no longer has an simplified closed form due to reneging. It is calculated as $B=1-\sum_{m+n\ge1}\pi^L_{m,n}-\sum_{m+n\ge1}\pi^R_{m,n}-\sum_{i\ge1, j\ge1}q_{i,j}$.} The proof of the theorem consists of checking that the conjectured steady state probabilities \eqref{eq:two_sided_mm1}, \eqref{eq:two_sided_left} and \eqref{eq:two_sided_right} satisfy the equilibrium equations \eqref{eq:NB2-st-1-L} to \eqref{eq:NB2-st-4-M}. This task is greatly simplified because of the aforementioned similarities between the one-sided system and the two-sided system. For example, consider the left side of the two-sided system for which we need to verify 
\eqref{eq:NB2-st-1-L} to \eqref{eq:NB2-st-3-L}. Note that 
both equation \eqref{eq:NB2-st-1-L} and \eqref{eq:NB2-st-3-L}  appear in the equilibrium equations for the one-sided system. Since $\pi_{m,n}^L$ has the same form as $\pi_{m,n}$, by using Theorem \ref{thm-prod-one-sided}, these two equilibrium equations are readily checked. \fcu{ For \eqref{eq:NB2-st-2-L}, 
we can leverage our one-sided results to obtain 
much simplified equations. Indeed, since $\pi^L_{m,n}$ satisfies Eq. \eqref{eq:NB-st-2} we deduce that  Eq. \eqref{eq:NB2-st-2-L} becomes}
\begin{equation*}
    q_{m,1}(\lambda_1 + \theta_d) = \pi^L_{m,0}\mu_1 + \sum_{k=1}^{m}\pi^L_{m-k,k}\mu_1(1-\gamma_s)^k,\quad m\geq 1,
\end{equation*}
which can be readily verified --- \cyu{after using the partial balance equation to simplify the second term on the right hand sided as we did for the one-sided case, see Eq. \eqref{eq:key-step}}. Since left and right sides are symmetric, a similar argument applies to the right side. 
Finally, the remaining equations \eqref{eq:NB2-st-1-M} to \eqref{eq:NB2-st-4-M} are much easier to verify as they very much behave as two separate M/M/1+M queues. We present the full details in the appendix. 

\section{Concluding Remarks}
\label{sec:conclusion}

In this paper we perform an exact analysis of the steady state probabilities for a matching queue with reneging specified by an N-system. We first investigate a one-sided system in which demand units leave the system if they are not immediately matched; then we analyze the two-sided case in which demand units would queue in the system. We note that our approach of solving for the steady state probabilities is constructive. It relies on identifying an intrinsic partial balance equation which equalizes the flow into state $(m,0)$ due to all unknown supply becoming inflexible supply with the flow out of the state due to inflexible supply leaving the system. {\revision Here we offer some concluding remarks and future directions.}

\textbf{Type-dependent reneging rate.} {\revision A natural extension of our work is to consider different reneging rates for type 1 and type 2 supply (demand), $\theta_{s,1}$ and $\theta_{s,2}$ ($\theta_{d,1}$ and $\theta_{d,2}$). An immediate challenge is that the parsimonious state space representation $(m,n)$ considered in this paper is no longer Markovian. This is because the $n$ part of the state space now consists of supply (demand) of different types, and the total reneging rates of these $n$ supplies (demands) cannot be determined using the memoryless property. Consider an example of only one supply with unknown type. The patience of that supply is distributed as the mixture of two exponential random variables with rates $\theta_{s,1}$ and $\theta_{s,2}$, which is no longer exponentially distributed. On the other hand, under the expanded state space considered in \cite{adan2018fcfs} which tracks the complete arrival sequence of supply of different types, one can derive a product-form solution by directly incorporating reneging rates (see Theorem 3.13 in \cite{gardner2020product}). An interesting direction could be developing  a new  state space representation which preserves Markovian properties under type-dependent reneging rates.}

\textbf{General matching graph.} The N-system studied in this paper has a wide variety of applications and it is also the cornerstone for analyzing more general matching queues. We believe our ideas can be useful in the analysis of more complex topologies. {\revision Note that a product form solution under general matching graphs with reneging can be obtained using the expanded state space representation in \cite{adan2018fcfs} and \cite{gardner2020product} which tracks the entire arrival sequence. However, this state representation might not be tractable for computation. Another approach is to use the multidimensional but reduced state space representation in \cite{visschers2012product}. That representation assumes that servers are fixed a priori, and partitions the arrival sequence into pieces separated by servers with different compatibility. In our context, however, servers are arriving future demands. Thus, the first step is to adapt the state space representation to our setting, where partitions have to be reformed upon demand arrival. Then, new structural product forms, in contrast to the form of  $\pi_{m,n}=f(m)g(m+n)\cdot B$ in the N-system, have to be investigated. On a related note, given our analysis of the two-sided matching queue, it is likely that the results of general matching graphs (once completed) can be carried over to two-sided systems as well.}

\textbf{Performance evaluation.}
There has also been a number of studies that consider the design and performance evaluation of parallel FCFS systems. The problem of how to optimally design a matching graph to balance waiting time delays and  matching rewards generated by pairing customers and servers under FCFS-ALIS has been studied in \cite{afeche2019optimal}.  \cite{adan2019design} design heuristics to approximate customer-to-sever matching rates by means of an FCFS bipartite infinite matching model first proposed in \cite{caldentey2007heavy}, and further analyzed in \cite{adan2012exact} and \cite{fazel2018approximating}. We hope that our work can provide some insights on how to design and analyze other bipartite matching systems where agents might abandon. 

\begin{acknowledgements}

We would like to thank Professor Gideon Weiss, Ivo Adan and two anonymous reviewers for comments and suggestions that has encouraged and improved this work. 

\end{acknowledgements}

\newpage
\appendix
\section{Additional Results}
\label{apx:additional_results}
\begin{proof}{\emph{of Proposition \ref{prop:no-reneg}}.} The set of equilibrium equations that needs to be verified are Eq. \eqref{eq:NB-st-1} to Eq. \eqref{eq:NB-st-4} where we replace $\sr$ with zero. To simplify notation define 
\begin{equation*}
    x=\frac{\lambda_2}{\mu_2},\quad y=\frac{\lambda_1+\lambda_2}{\mu_1+\mu_2}, \quad \text{and}\quad C=\frac{\mu_1}{\mu_1+\mu_2}.
\end{equation*}
The key step is to compute 
\begin{equation*}
    \sum_{k=0}^m\pi_{m-k,n+k+1}\mu_1\gamma_s(1-\gamma_s)^k\quad \text{and}\quad 
    \sum_{k=1}^{m}\pi_{m-k,k}\mu_1(1-\gamma_s)^k.
\end{equation*}
For the first we have (for $n\ge 0$)
\begin{flalign*}
\frac{1}{B}\sum_{k=0}^m\pi_{m-k,n+k+1}\mu_1\gamma_s(1-\gamma_s)^k&=
C\sum_{k=0}^{m-1} x^{m-k}y^{n+k+1}\mu_1\gamma_s(1-\gamma_s)^k+
 y^{n+m+1}\mu_1\gamma_s(1-\gamma_s)^m\\
&= C\sum_{k=0}^{m-1} \left(\frac{x}{y}\right)^{m-k}y^{m+n+1}\mu_1\gamma_s(1-\gamma_s)^k+
 y^{n+m+1}\mu_1\gamma_s(1-\gamma_s)^m \\
&=
 y^{n+1}\mu_1\gamma_s\left(C x^m\sum_{k=0}^{m-1} \left(\frac{x}{y}\right)^{-k}(1-\gamma_s)^k+
 y^{m}(1-\gamma_s)^m\right)\\
 &=
 y^{n+1}\mu_1\gamma_s\left(C x^m \frac{1-\left(\frac{y(1-\gamma_s)}{x}\right)^m}{1-\frac{y(1-\gamma_s)}{x}}+
 y^{m}(1-\gamma_s)^m\right)\\
 &=y^{n+1}\mu_1\gamma_s x^m\\
 &=C\lambda_1 y^{n} x^m,\\
\end{flalign*}
where the last two equalities come from replacing the values of $x,y$ and $C$. Similarly, we can verify that 
\begin{equation*}
    \frac{1}{B}\sum_{k=1}^{m}\pi_{m-k,k}\mu_1(1-\gamma_s)^k=C\mu_2 x^m.
\end{equation*}
Next we proceed to verify the equilibrium equations.  Eq. \eqref{eq:NB-st-1} and 
Eq. \eqref{eq:NB-st-2} can be simplified by replacing the result derived above
which deliver
\begin{flalign*}
    Cy^{n} x^m\left(\mu_1+\mu_2+\lambda_1+\lambda_2\right)
&= Cy^{n} x^{m+1}\mu_2+Cy^{n-1} x^m(\lambda_1+\lambda_2)
+C\lambda_1 y^{n} x^m\\
C x^m(\mu_2+\lambda_1+\lambda_2)&=  C y x^{m+1}\mu_2 
 + C\lambda_1 y x^m +C\mu_2 x^m,
\end{flalign*}
for Eq. \eqref{eq:NB-st-1} and 
Eq. \eqref{eq:NB-st-2}, respectively. After dividing both sides (in both equations) by $Cx^m y^n$, it is easy to see that the relations above are satisfied. Eq. \eqref{eq:NB-st-3} and
Eq. \eqref{eq:NB-st-4} can be verified by a similar procedure. Moreover,
note that $B$ can be explicitly computed 
\begin{equation*}
    B=\frac{(\mu_1+\mu_2-\lambda_1-\lambda_2)(\mu_2-\lambda_2)}{(\mu_1+\mu_2-\lambda_2)\mu_2}.
\end{equation*}
Finally, note that by Proposition 6 in \cite{visschers2012product}
the stability conditions in the statement of the proposition guarantee ergodicity (and that $B$ is well defined). 
In turn, $\pi_{m,n}$ corresponds to the steady state probabilities. 

\end{proof}
\begin{proof}{\emph{of Theorem \ref{thm-prod-one-sided}}.}
The first step in the proof is to rewrite the steady state probabilities \cyu{in the form of $f(m)g(m+n)$}. With this new characterization, we then proceed to verify equations \eqref{eq:NB-st-1} to \eqref{eq:NB-st-4}.

According to Lemma \ref{lm:alternative-form}, we can cast $\pi_{m,n}$ as given in the statement of the proposition as 
\begin{align*}
\pi_{m,n}
=\left(\prod_{i=1}^{m+n}\frac{\lambda_1+\lambda_2}{\mu_1+\mu_2+i \theta_s}\right)\cdot 
\left(\prod_{i=1}^{m}\frac{(\mu_1+\mu_2\1{i>1}+(i-1)\theta_s)(1-\gamma_s)}{\mu_2+i \theta_s}\right) B,\quad \forall m,n\ge 0.
\end{align*}

Now note that in the discussion after Theorem \ref{thm-prod-one-sided} we observe that the only remaining pieces that need to be verified are:
\begin{enumerate}
    \item[\textbf{(a)}] Eq. \eqref{eq:NB-st-2}
    \item[\textbf{(b)}] Eq. \eqref{eq:NB-st-1}
\end{enumerate}

We begin with \textbf{(a)}. For ease of notation define,
\begin{align*}
p_i = \frac{\lambda_1+\lambda_2}{\mu_1+\mu_2+i \theta_s},\quad g(n)=\prod_{i=1}^np_i,
\quad \text{and} \quad
s_i =\frac{(\mu_1+\mu_2\1{i>1}+(i-1)\theta_s)(1-\gamma_s)}{\mu_2+i \theta_s},\quad
f(m)=\prod_{i=1}^m s_i,
\end{align*}
observe that the definition of $g(n)$ and $f(m)$ are the same as in the discussion after the statement of Theorem \ref{thm-prod-one-sided}.

Thus $\pi_{m,n}$ can then be summarized and re-written as
\begin{equation*}
    \pi_{m,n} = g(m+n)f(m)B.
\end{equation*}
 Eq. \eqref{eq:NB-st-2} then becomes
\begin{flalign*}
g(m)f(m)(\mu_2+\lambda_1+\lambda_2 + m\theta_s) &=g(m)f(m)p_{m+1}\theta_s
+ g(m)f(m)p_{m+1}s_{m+1}(\mu_2 + \theta_s(m + 1)) \\
&+\sum_{k=0}^mg(m)f(m-k)p_{m+1}\mu_1\gamma_s(1-\gamma_s)^k + \sum_{k=1}^{m}g(m)f(m-k)\mu_1(1-\gamma_s)^k.
\end{flalign*}

Dividing both sides of the equation by $g(m)$ this becomes,
\begin{flalign}
\nonumber
f(m)(\mu_2+\lambda_1+\lambda_2 + m\theta_s) &=f(m)p_{m+1}\theta_s
+ f(m)p_{m+1}s_{m+1}(\mu_2 + \theta_s(m + 1))\\ \nonumber
&+\sum_{k=0}^mf(m-k)p_{m+1}
\mu_1\gamma_s(1-\gamma_s)^{k} 
+\sum_{k=1}^{m}f(m-k)\mu_1(1-\gamma_s)^k\\ \nonumber
&=f(m)p_{m+1}(\theta_s+\mu_1\gamma_s+(\mu_1+\mu_2 + m\theta_s)(1-\gamma_s))\\ \label{eq:intermediate}
&+ \mu_1(\gamma_s p_{m+1}+1)\sum_{k=1}^mf(m-k)(1-\gamma_s)^{k}.
\end{flalign}

To ease the derivation, we use the following simplification result, which is formally proved in Lemma \ref{lm:key-step-induction},

\begin{equation}
\label{eq:simplification}
    \sum_{k=1}^m f(m-k)(1-\gamma_s)^{k} =
    \frac{f(m)(\mu_2+m\cdot\sr)}{\mu_1}. 
\end{equation}

Next, we simplify Eq. \eqref{eq:intermediate} using Eq. \eqref{eq:simplification}.
\begin{align*}
f(m)p_{m+1}(\theta_s+\mu_1\gamma_s+(\mu_1+\mu_2 + m\theta_s)(1-\gamma_s))\\
+ \mu_1(\gamma_s p_{m+1}+1)(1-\gamma_s)^m\sum_{k=1}^m f(m-k)(1-\gamma_s)^{k-m} 
=&f(m)p_{m+1}(\theta_s+\mu_1\gamma_s+(\mu_1+\mu_2 + m\theta_s)(1-\gamma_s))\\
+& \mu_1(\gamma_s p_{m+1}+1)(1-\gamma_s)^mf(m)\frac{\mu_2 + m\theta_s}{\mu_1(1-\gamma_s)^m}\\
=&f(m)\big(p_{m+1}(\mu_1+\mu_2+(m+1)\theta_s) - \gamma_sp_{m+1}(\mu_2 + m\theta_s))\big) \\
+&f(m)\big((\gamma_sp_{m+1}+1)(\mu_2 + m\theta_s)\big) \\
=&f(m)(\lambda_1 + \lambda_2 + \mu_2 + m\theta_s),
\end{align*}
thus we have verified Eq. \eqref{eq:NB-st-2}. 

To conclude, we move to  \textbf{(b)} and  verify Eq. \eqref{eq:NB-st-1}. This is equivalent to check that the following equation holds:
\begin{flalign*}
g(m+n)f(m)(\mu_1+\mu_2+\lambda_1+\lambda_2 + (m+n) \theta_s ) 
&=g(m+n)p_{m+n+1}f(m) ((n+1)\theta+\mu_1\gamma_s)\\
&+p_{m+n+1}S_{m+1}(\mu_2+(m+1)\theta_s)\\
&+g(m+n)p_{m+n+1}\gamma_s\underbrace{~\sum_{k=1}^m f(m-k)\mu_1(1-\gamma_s)^k}_{ f(m)(\mu_2+m\theta_s)\textrm{~by Eq. \eqref{eq:simplification}}} \\
  &+\frac{g(m+n)f(m)}{p_{m+n}}(\lambda_1+\lambda_2)
\end{flalign*}
Dividing both sides of the equation by $g(m+n)f(m)$, we get
\begin{align*}
(\mu_1+\mu_2+\lambda_1+\lambda_2 + (m+n) \theta_s ) 
&=p_{m+n+1} ((n+1)\theta_s+\mu_1\gamma_s)
+p_{m+n+1}s_{m+1}(\mu_2+(m+1)\theta_s)\\
&+p_{m+n+1}\gamma_s(\mu_2 + m\theta_s)
  +\frac{\lambda_1+\lambda_2}{p_{m+n}} \\
&=p_{m+n+1} ((n+1)\theta_s+\mu_1\gamma_s) + p_{m+n+1}(\mu_1 + \mu_2 +m\theta_s)(1-\gamma_s)\\
&+p_{m+n+1}\gamma_s(\mu_2 + m\theta_s)
  +\mu_1 +\mu_2 +(m+n)\theta_s \\
&=p_{m+n+1}((m+n+1)\theta_s + \mu_1 + \mu_2) + \mu_1 + \mu_2 + (m+n)\theta_s \\
&=\lambda_1 + \lambda_2 + \mu_1 + \mu_2 + (m+n)\theta_s,
\end{align*}
which completes the verification. $\qed$
\end{proof}

\begin{lemma}\label{lm:alternative-form}
The expressions
\begin{equation*}
    \left\{\begin{array}{ll}\left(\frac{\mu_1}{\mu_1 + \mu_2 + m\theta_s }\right)\left(\prod_{i=1}^m\frac{\lambda_2}{\mu_2 + i\theta_s}\right)\left(\prod_{i=1}^n\frac{\lambda_1+\lambda_2}{\mu_1 + \mu_2 + m\theta_s + i\theta_s}\right), & \quad m\ge1 \\[1em]
    \prod_{i=1}^n \frac{\lambda_1+\lambda_2}{\mu_1+\mu_2+i\theta_s}, & \quad m=0
    \end{array}\right.
\end{equation*}
and 
\begin{equation*}
    \left(\prod_{i=1}^{m+n}\frac{\lambda_1+\lambda_2}{\mu_1+\mu_2+i \theta_s}\right)\cdot 
\left(\prod_{i=1}^{m}\frac{(\mu_1+\mu_2\1{i>1}+(i-1)\theta_s)(1-\gamma_s)}{\mu_2+i \theta_s}\right),\quad \forall m,n\ge 0
\end{equation*}
are equivalent. 
\end{lemma}
\begin{proof}
Denote the first term in the statement $L_1$ and the second $L_2$.

For $m\ge1$,
\begin{align*}
L_1&=\bigg(\frac{\mu_1}{\mu_1 + \mu_2 + m\theta_s }\bigg)\cdot\bigg(\prod_{i=1}^m\frac{\lambda_2}{\mu_2 + i\theta_s}\bigg)\cdot\bigg(\prod_{i=1}^n\frac{\lambda_1+\lambda_2}{\mu_1 + \mu_2 + m\theta_s + i\theta_s}\bigg)  \\
&=\bigg(\frac{\mu_1}{\mu_1 + \mu_2 + m\theta_s }\bigg)\cdot\bigg(\prod_{i=1}^m\frac{(\lambda_1 + \lambda_2)(1-\gamma_s)}{\mu_2 + i\theta_s}\bigg) \cdot\bigg(\prod_{i=1}^n\frac{\lambda_1+\lambda_2}{\mu_1 + \mu_2 + m\theta_s + i\theta_s}\bigg)\cdot\underbrace{\bigg(\prod_{i=1}^m\frac{\mu_1+\mu_2 + i\theta_s}{\mu_1+\mu_2 + i\theta_s}\bigg)}_{=1} \\
&=\bigg(\frac{\mu_1}{\mu_1 + \mu_2 + m\theta_s }\bigg)\cdot \left(\prod_{i=1}^{m+n}\frac{\lambda_1+\lambda_2}{\mu_1+\mu_2+i \theta_s}\right)
\cdot \bigg(\prod_{i=1}^m\frac{(\mu_1+\mu_2 + i\theta_s)(1-\gamma_s)}{\mu_2 + i\theta_s}\bigg)
\\
&=\left(\prod_{i=1}^{m+n}\frac{\lambda_1+\lambda_2}{\mu_1+\mu_2+i \theta_s}\right)\cdot 
\left(\prod_{i=1}^{m}\frac{(\mu_1+\mu_2\1{i>1}+(i-1)\theta_s)(1-\gamma_s)}{\mu_2+i \theta_s}\right) \\
&=L_2.
\end{align*}

For $m=0$,
\begin{align*}
    L_1 = \prod_{i=1}^n\frac{\lambda_1+\lambda_2}{\mu_1+\mu_2+i\theta_s}=L_2.
\end{align*}

This completes the proof.
\end{proof}

\begin{lemma}\label{lm:key-step-induction}
Let $f:\mathbb{N}\rightarrow\mathbb{R}_+$ be a function such that $f(0)=1$ and 
 \begin{equation}\label{eq1:lem-form}
 \sum_{k=1}^{m}f(m-k)(1-\gamma_s)^k=f(m)\cdot (a+m\cdot b),\quad \forall m\ge 1.
 \end{equation}
 Then
  \begin{equation}\label{eq2:lem-form}
 f(m)=\frac{(1-\gamma_s)^m}{a+b}\left(\prod_{i=2}^m \frac{1+a+(i-1)b}{a+ib}\right),\quad \forall m\geq1.
 \end{equation}
\end{lemma}
\begin{proof}
We proceed by induction. For $m=1$ since $f(0)=1$ Eq.\eqref{eq1:lem-form} yields
\begin{equation}
f(1)(a+b)=(1-\gamma_s),
\end{equation}
which coincides with Eq. \eqref{eq2:lem-form}. For the induction step let us compute $f(m)$ assuming that the property holds form $f(m-1)$. From Eq. \eqref{eq2:lem-form} we have  
\begin{flalign*}
f(m)(a+m\cdot b)&=
\sum_{k=1}^{m}f(m-k)(1-\gamma_s)^k\\
&=\sum_{k=0}^{m-1}f(m-1-k)(1-\gamma_s)^{k+1}\\
&=f(m-1)(1-\gamma_s)+(1-\gamma_s)\sum_{k=1}^{m-1}f(m-1-k)(1-\gamma_s)^{k}\\
&=f(m-1)(1-\gamma_s)+(1-\gamma_s)f(m-1)(a+(m-1)\cdot b)\\
&=f(m-1)(1-\gamma_s)(1+a+(m-1)\cdot b)\\
&=\frac{(1-\gamma_s)^{m-1}}{a+b}\left(\prod_{i=2}^{m-1} \frac{1+a+(i-1)b}{a+ib}\right)(1-\gamma_s)(1+a+(m-1)\cdot b)\\
&=\frac{(1-\gamma_s)^m}{a+b}\left(\prod_{i=2}^m \frac{1+a+(i-1)b}{a+ib}\right)\cdot(a+m\cdot b),
\end{flalign*}
where the second to last equality comes from the induction hypothesis and Eq. \eqref{eq2:lem-form}. This concludes the proof. 
\end{proof}

\begin{proof}{\emph{of Theorem \ref{thm:two-sided}}.}
First we rewrite the expression for the conjectured steady state probabilities $\pi^L_{m,n}$ and $\pi^R_{m,n}$ into the form of $\pi_{m,n}=f(m)g(m+n)B$ as follows (see Lemma \ref{lm:alternative-form} for a formal proof),

\begin{align*}
\pi_{m,n}^L &= \left(\prod_{i=1}^{m+n}\frac{\lambda_1+\lambda_2}{\mu_1+\mu_2+i \theta_s}\right)
\left(\prod_{i=1}^{m}\frac{(\mu_1+\mu_2\1{i>1}+(i-1)\theta_s)(1-\gamma_s)}{\mu_2+i \theta_s}\right)B,\\
\pi_{m,n}^R &= \left(\prod_{i=1}^{m+n}\frac{\mu_1+\mu_2}{\lambda_1+\lambda_2+i \theta_d}\right)
\left(\prod_{i=1}^{m}\frac{(\lambda_2+\lambda_1\1{i>1}+(i-1)\theta_d)(1-\gamma_d)}{\lambda_1+i \theta_d}\right)B.
\end{align*}

For ease of notation, for $\ell\in \{L,R\}$ we define 
\begin{equation*}
P^\ell_{n}=\prod_{i=1}^n p^\ell_i,
\quad 
\text{where}\quad
p^L_i = \frac{\lambda_1+\lambda_2}{\mu_1+\mu_2+i \theta_s},\quad
p^R_i = \frac{\mu_1+\mu_2}{\lambda_1+\lambda_2+i \theta_d},\quad i\geq 1.
\end{equation*}
and 
\begin{equation*}
S^\ell_{m}=\prod_{i=1}^m s^\ell_i,
\end{equation*}
where
\begin{align*}
s^L_i =\frac{(\mu_1+\mu_2\1{i>1}+(i-1)\theta_s)(1-\gamma_s)}{\mu_2+i \theta_s},\quad s^R_i =\frac{(\lambda_2+\lambda_1\1{i>1}+(i-1)\theta_d)(1-\gamma_d)}{\lambda_1+i \theta_d},\quad i\geq 1.
\end{align*}
This gives the following forms of $\pi^L_{m,n}$ and $\pi^R_{m,n}$:
\begin{equation*}
\pi^L_{m,n} = P^L_{m+n}S^L_m,~~\pi^R_{m,n} = P^R_{m+n}S^R_m.\\
\end{equation*}
Now we verify the equilibrium equations. To do so we will make use of  Theorem \ref{thm-prod-one-sided}.
Observe that from Theorem \ref{thm-prod-one-sided}, equations \eqref{eq:NB2-st-1-L}, \eqref{eq:NB2-st-3-L}, \eqref{eq:NB2-st-1-R} and \eqref{eq:NB2-st-3-R} are readily verified.

Next we check Eq. \eqref{eq:NB2-st-2-L}. According to Theorem \ref{thm-prod-one-sided}, $\pi^L_{m,n}$ satisfies Eq. \eqref{eq:NB-st-2}. Adding $\pi^L_{m,0}\mu_1$ on both sides of Eq. \eqref{eq:NB-st-2}, we have
\begin{align*}
\pi^L_{m,0}(\mu_1 + \mu_2+\lambda_1+\lambda_2 + m\theta_s) &= \pi^L_{m,0}\mu_1 + \pi^L_{m,1}\theta_s
+ \pi^L_{m+1,0}\left(\mu_2 + \theta_s(m + 1)\right) \\
&+ \sum_{k=0}^m\pi^L_{m-k,k+1}\mu_1\gamma_s(1-\gamma_s)^k +\sum_{k=1}^{m}\pi^L_{m-k,k}\mu_1(1-\gamma_s)^k.
\end{align*}
In order to verify Eq. \eqref{eq:NB2-st-2-L}, we just need to show that
\begin{equation*}
    q_{m,1}(\lambda_1 + \theta_d) = \pi^L_{m,0}\mu_1 + \sum_{k=1}^{m}\pi^L_{m-k,k}\mu_1(1-\gamma_s)^k.
\end{equation*}
The left hand side $q_{m,1}(\lambda_1 + \theta_d) = \left(\prod_{i=1}^m\frac{\lambda_2}{\mu_2+i\theta_s}\right)\mu_1B$. For the right hand side,
\begin{align}
\nonumber
\pi^L_{m,0}\mu_1 + \sum_{k=1}^{m}\pi^L_{m-k,k}\mu_1(1-\gamma_s)^k 
&=P^L_m S^L_m\mu_1B + P^L_{m}B\underbrace{\sum_{k=1}^{m} S^L_{m-k}\mu_1(1-\gamma_s)^k}_{S^L_m(\mu_2 + m\theta_s)\textrm{~by Lemma \ref{lm:key-step-induction}}} \\ \nonumber
&=P_m^LS_m^LB(\mu_1+\mu_2+m\theta_s) \\ \nonumber
&=\left(\prod_{i=1}^m \frac{\lambda_1 + \lambda_2}{\mu_1 + \mu_2 + i\theta_s}\right)\left(\prod_{i=1}^m \frac{(\mu_1+\mu_2\1{i>1}+(i-1)\theta_s)(1-\gamma_s)}{\mu_2+i \theta_s}\right) \\ \nonumber
&\cdot(\mu_1+\mu_2+m\theta_s)B \\ \nonumber
&=\left(\prod_{i=1}^m \frac{\lambda_2}{\mu_2 + i\theta_s}\right)\left(\prod_{i=1}^m \frac{(\mu_1+\mu_2\1{i>1}+(i-1)\theta_s)}{\mu_1 + \mu_2+i \theta_s}\right) (\mu_1+\mu_2+m\theta_s)B \\ \label{eq:simplification2}
&=\left(\prod_{i=1}^m \frac{\lambda_2}{\mu_2 + i\theta_s}\right)\mu_1 B,
\end{align}
which completes the verification. Eq. \eqref{eq:NB2-st-2-R} can be verified with the same technique.\\

We now verify equations from \eqref{eq:NB2-st-1-M} to \eqref{eq:NB2-st-4-M}. We start with \eqref{eq:NB2-st-1-M}. For $i,j\geq 2$, dividing $q_{i,j}$ on both sides of Eq. \eqref{eq:NB2-st-1-M} we get the first equality of the following:
\begin{flalign*}
(i\theta_s + j\theta_d + \mu_1 + \mu_2 + \lambda_1 + \lambda_2 ) 
&= \frac{\lambda_2}{\mu_2+(i+1)\theta_s}((i+1)\theta_s + \mu_2) + 
\frac{{\mu_2+i\theta_s}}{\lambda_2}\lambda_2\\
&+ \frac{\mu_1}{\lambda_1+(j+1)\theta_d}((j+1)\theta_d + \lambda_1) + 
\frac{\lambda_1+j\theta_d}{\mu_1}\mu_1\\
&=\lambda_2 + \mu_2 + i\theta_s + \mu_1 + \lambda_1 + j\theta_d,
\end{flalign*}
which verifies Eq. \eqref{eq:NB2-st-1-M}. \\

For Eq. \eqref{eq:NB2-st-2-M}, replacing $\pi_{j-k,k}^R$ with $P_j^RS_{j-k}^R$ we have the first equality of the following,
\begin{flalign*}\nonumber
q_{1,j}(\theta_s + j\theta_d + \mu_1 + \mu_2 + \lambda_1 + \lambda_2) 
&= q_{2,j}(2\theta_s + \mu_2) + q_{1,j+1}((j+1)\theta_d + \lambda_1) + q_{1,j-1}\mu_1\\ 
&+ BP_j^RS_j^R\lambda_2 + BP^R_j\underbrace{\sum_{k=1}^jS^R_{j-k}\lambda_2(1-\gamma_d)^k}_{S^R_j(\lambda_1 + m\theta_d)\textrm{~by Lemma \ref{lm:key-step-induction}}} \\
&= q_{2,j}(2\theta_s + \mu_2) + q_{1,j+1}((j+1)\theta_d + \lambda_1) + q_{1,j-1}\mu_1\\ 
&+ BP_j^RS_j^R(\lambda_1+\lambda_2+m\theta_d) \\
&= q_{2,j}(2\theta_s + \mu_2) + q_{1,j+1}((j+1)\theta_d + \lambda_1) + q_{1,j-1}\mu_1\\ 
&+ B\lambda_2\left(\prod_{k=1}^j\frac{\mu_1}{\lambda_1 + k\theta_d}\right),\quad\textrm{by Eq. \eqref{eq:simplification2}}.
\end{flalign*}
Dividing both sides of the equation by $q_{1,j}$, we get
\begin{align*}
\theta_s + j\theta_d + \mu_1 + \mu_2 + \lambda_1 + \lambda_2 
&=\frac{\lambda_2}{\mu_2 + 2\theta_s}(2\theta_s + \mu_2) + \frac{\mu_1}{\lambda_1 + (j+1)\theta_d}((j+1)\theta_d + \lambda_1)  \\ 
&+ \frac{\lambda_1 + j\theta_d}{\mu_1}\mu_1 + \frac{\mu_2+\theta_s}{\lambda_2}\lambda_2 \\
&=\lambda_2 + \mu_1 + \lambda_1 + j\theta_d + \mu_2 + \theta_s.
\end{align*}
This verifies Eq. \eqref{eq:NB2-st-2-M}. By symmetry we can also conclude that \eqref{eq:NB2-st-3-M} is verified as well. \\

For Eq. \eqref{eq:NB2-st-4-M}, since $q_{2,1}(2\theta_s + \mu_2) = q_{1,1}\lambda_2, q_{1,2}(2\theta_d + \lambda_1) = q_{1,1}\mu_1$, we can cancel out $q_{1,1}(\mu_1+\lambda_2)$ on both sides of Eq. \eqref{eq:NB2-st-4-M}. This gives the following equation,
\begin{flalign*}
q_{1,1}(\theta_s + \theta_d  + \mu_2 + \lambda_1 )&= \pi^R_{1,0}\lambda_2 + \pi^R_{0,1}\lambda_2(1-\gamma_d)+ \pi^L_{1,0}\mu_1 + \pi^L_{0,1}\mu_1(1-\gamma_s).
\end{flalign*}
Replacing $q_{1,1}, \pi_{1,0}^R, \pi_{0,1}^R, \pi_{1,0}^L, \pi_{0,1}^L$ with their corresponding expressions, we get the first equality of the following,
\begin{align*}
\frac{\lambda_2}{\mu_2+\theta_s}\frac{\mu_1}{\lambda_1+\theta_d}(\theta_s + \theta_d  + \mu_2 + \lambda_1 ) 
&=\frac{\lambda_2(1-\gamma_d)}{\lambda_1+\theta_d}
\frac{\mu_1+\mu_2}{\lambda_1+\lambda_2+\theta_d}\lambda_2 + \frac{\mu_1+\mu_2}{\lambda_1+\lambda_2+\theta_d}\lambda_2(1-\gamma_d)\\
&+ \frac{\mu_1(1-\gamma_s)}{\mu_2+\theta_s}\frac{\lambda_1+\lambda_2}{\mu_1+\mu_2+\theta_s}\mu_1 + \frac{\lambda_1+\lambda_2}{\mu_1+\mu_2+\theta_s}\mu_1(1-\gamma_s) \\
&=\frac{\mu_1+\mu_2}{\lambda_1+\lambda_2+\theta_d}\lambda_2(1-\gamma_d)\frac{\lambda_1+\lambda_2+\theta_d}{\lambda_1+\theta_d} \\
&+\frac{\lambda_1+\lambda_2}{\mu_1+\mu_2+\theta_s}\mu_1(1-\gamma_s)\frac{\mu_1+\mu_2+\theta_s}{\mu_2+\theta_s} \\
&=\frac{\lambda_2\mu_1}{\lambda_1+\theta_d} + \frac{\lambda_2\mu_1}{\mu_2 + \theta_s}.
\end{align*}
This in turn verifies Eq. \eqref{eq:NB2-st-4-M}.\\

Finally, we verify Eq. \eqref{eq:NB2-st-4-L}. Replacing $\pi_{0,1}^L, \pi_{1,0}^L, \pi_{0,1}^R, \pi_{1,0}^R$ with their corresponding terms, we have the first equality of the following,
\begin{align*}
(\lambda_1+\lambda_2+\mu_1+\mu_2)
&=
\frac{\lambda_1+\lambda_2}{\mu_1+\mu_2+\theta_s}(\theta_s + \mu_1\gamma_s + \mu_2) + 
\frac{\lambda_1+\lambda_2}{\mu_1+\mu_2+ \theta_s}\frac{\mu_1(1-\gamma_s)}{\mu_2+\theta_s}(\mu_2 + \theta_s)\\
&+\frac{\mu_1+\mu_2}{\lambda_1+\lambda_2+\theta_d}
\frac{\lambda_2(1-\gamma_d)}{\lambda_1+ \theta_d}
(\theta_d+ \lambda_1)+
\frac{\mu_1+\mu_2}{\lambda_1+\lambda_2+\theta_d}(\theta_d + \lambda_1 + \lambda_2\gamma_d) \\
&=\frac{\lambda_1+\lambda_2}{\mu_1+\mu_2+\theta_s}(\mu_1+\mu_2+\theta_s) + \frac{\mu_1+\mu_2}{\lambda_1+\lambda_2+\theta_d}(\lambda_1 + \lambda_2 + \theta_d) \\
&=\lambda_1+\lambda_2+\mu_1+\mu_2,
\end{align*}
which verifies Eq. \eqref{eq:NB2-st-4-L}. This concludes the proof. $\qed$
\end{proof}

\bibliographystyle{plainnat}
\bibliography{references.bib}   %

\end{document}